\documentclass[a4paper,10pt]{article}

\usepackage{ucs} 
\usepackage[utf8x]{inputenc}
\usepackage{graphicx}
\usepackage{amsfonts}
\usepackage{dsfont}
\usepackage{amssymb}
\usepackage{amsmath}
\usepackage{amsthm}
\usepackage{enumerate}
\usepackage{stmaryrd}
\usepackage{fullpage}
\usepackage[colorlinks=true]{hyperref}
\usepackage{ifthen}
\usepackage{subfigure}
\usepackage{authblk}

\newcommand{\R}{\mathbb{R}}
\newcommand{\D}{\mathbb{D}}

\newcommand{\mbW}{\mathbf{W}}
\newcommand{\mV}{\mathbf{V}}

\newcommand{\mcV}{\mathcal{V}}

\newcommand{\T}{\mathcal{T}}
\newcommand{\tT}{\tilde\T}

\newcommand{\mF}{\mathcal{F}}
\newcommand{\mO}{\mathcal{O}}
\newcommand{\mG}{\mathcal{G}}
\newcommand{\mK}{\mathcal{K}}
\newcommand{\mC}{\mathcal{C}}
\newcommand{\mB}{\mathcal{B}}
\newcommand{\mM}{\mathcal{M}}

\newcommand{\mR}{\mathcal{R}}

\newcommand{\mW}{\mathcal{W}}
\newcommand{\bq}{\begin{equation}}
\newcommand{\eq}{\end{equation}}
\newcommand{\bqs}{\begin{equation*}}
\newcommand{\eqs}{\end{equation*}}

\newboolean{DisplaySOS}
\setboolean{DisplaySOS}{true} 
\newcommand{\SOS}[1]{\ifthenelse{\boolean{DisplaySOS}}{{\textcolor{red}{\bf[#1]}}}{}}

\graphicspath{{figures_hypth/}}

\title{Some theoretical results for a class of neural mass equations}
\author[1]{Gr\'egory Faye}
\author[1,2]{Pascal Chossat}
\author[1]{Olivier Faugeras}
\affil[1]{\small NeuroMathComp Laboratory, INRIA, Sophia Antipolis, CNRS, ENS Paris, France}
\affil[2]{\small Dept. of Mathematics, University of Nice Sophia-Antipolis, JAD Laboratory and CNRS, Parc Valrose, 06108 Nice Cedex 02, France}
	
\begin{document}

\maketitle

\begin{abstract}
We study the neural field equations introduced by Chossat and Faugeras in \cite{chossat-faugeras:09} to  model the representation and the processing of image edges and textures in the hypercolumns of the cortical area V1. The key entity, the structure tensor, intrinsically lives in a non-Euclidean, in effect hyperbolic, space. Its spatio-temporal behaviour is governed by nonlinear integro-differential equations defined on the Poincaré disc model of the two-dimensional hyperbolic space. Using methods from the theory of functional analysis we show the existence and uniqueness of a solution of these equations. In the case of stationary, i.e. time independent, solutions we perform a stability analysis which yields important results on their behavior. We also present an original study, based on non-Euclidean, hyperbolic,  analysis, of a spatially localised bump solution in a limiting case. We illustrate our theoretical results with numerical simulations.
\end{abstract}

{\noindent \bf Keywords:} 
Neural fields; nonlinear integro-differential equations; functional analysis; non-Euclidean analysis; stability analysis; hyperbolic geometry; hypergeometric functions; bumps.\\

{\noindent \bf AMS subject classifications:} 
30F45, 33C05, 34A12, 34D20, 34D23, 34G20, 37M05, 43A85, 44A35, 45G10, 51M10, 92B20, 92C20.
\section{Introduction}

Chossat and Faugeras in \cite{chossat-faugeras:09} have introduced a new and elegant approach to model the processing of image edges and textures in the hypercolumns of area V1 that is based on a nonlinear representation of the image first order derivatives called the structure tensor. They assumed that this structure tensor was represented by neuronal populations in the hypercolumns of V1 that can be described by equation similar to those proposed by Wilson and Cowan \cite{wilson-cowan:72}.\\
Our investigations are motivated by the work of Bressloff, Cowan, Golubitsky, Thomas and Wiener \cite{bressloff-cowan-etal:01,bressloff-cowan-etal:02} on the spontaneous occurence of hallucinatory patterns under the influence of psychotropic drugs and the further studies of Bressloff and Cowan \cite{bressloff-cowan:02,bressloff-cowan:02b,bressloff-cowan:03}. We hardly think that the natural spatial extension of our model would lead to an exciting anlysis of hyperbolic hallucinatory patterns. But, this requires first to better understand the a-spatial model and this is the subject of this present work. The a-spatial model can also be linked to the work by Ben-Yishai \cite{ben-yishai-bar-or-etal:95} and Hansel, Sompolinsky \cite{hansel-sompolinsky:97} on the ring model of orientation.\\
The aim of this paper is to present a general and rigorous mathematical framework for the modeling of neuronal populations in one hypercolumn of V1 by the structure tensor which is based on miscellaneous tools of functional analysis. We illustrate our results with numerical experiments. In section \ref{section:model} we briefly introduce the equations, in section \ref{section:exun} we analyse the problem of the existence and uniqueness of their solutions. In section \ref{section:stationary} we deal with stationary solutions. In section \ref{section:bump}, we present an analysis of what we called a hyperbolic radially symmetric stationary-pulse in a limiting case. In the penultimate, we present some numerical simulations of the solutions. We conclude in section \ref{section:conclusion}.

\section{The model}\label{section:model}
We recall that the structure tensor is a way of representing the edges and textures of a 2D image $\mathcal{I}(x,y)$ \cite{bigun-granlund:87,knutsson:89}. Moreover, a structure tensor can be seen as a $2\times2$ symmetric positive matrix.\\
We assume that a hypercolumn of V1 can represent the structure tensor in the receptive field of its neurons as the average membrane potential values of some of its membrane pouplations. Let $\T$ be a structure tensor. The average potential $V(\T,t)$ of the column has its time evolution that is governed by the following neural mass equation adapted from \cite{chossat-faugeras:09} where we allow the connectivity function $W$ to depend upon the time variable $t$ and we integrate over the set of $2\times2$ symmetric definite-positive matrices:
\bq \left\{ \begin{array}{lcl}
\partial_t V(\T, t)&=&-\alpha V(\T, t)+\int_{\text{SPD(2)}}W(\T,\T',t)S(V(\T',t))d\T'+I(\T,t)\quad \forall t > 0\\
V(\T,0)&=&V_0(\T)
\end{array}
\right.
\label{eq:neurmass}
\eq
 The nonlinearity $S$ is a sigmoidal function which may be expressed as:
\bqs
S(x)=\dfrac{1}{1+e^{-\mu x}}
\eqs
where $\mu$ describes the stiffness of the sigmoid. $I$ is an external input.\\
The set $\text{SPD(2)}$ is the set of $2\times2$ symmetric positive-definite matrices. It can be seen as a foliated manifold by way of the set of special symmetric positive definite matrices $\text{SSPD(2)}=\text{SPD(2)}\cap \text{SL}(2,\R)$. Indeed, we have: $\text{SPD(2)}\overset{hom}{=}\text{SSPD(2)}\times\R^{+}_{*}$. Furthermore, $\text{SSPD(2)}\overset{isom}{=}\D$, where $\D$ is the Poincare Disk, see e.g. \cite{chossat-faugeras:09}. As a consequence we use the following foliation of $\text{SPD(2)}$: $\text{SPD(2)}\overset{hom}{=}\D\times\R^{+}_{*}$, which allows us to write for all $\T\in\text{SPD(2)}$, $\T=(z,\Delta)$ with $(z,\Delta)\in\D\times\R^{+}_{*}$. $\T$, $z$ and $\Delta$ are related by the relation  $\det(\T)=\Delta^{2}$ and the fact that  $z$ is the representation in $\D$ of $\tilde{\T} \in \text{SSPD(2)}$ with $\T=\Delta\tilde{\T}$. \\
It is well-known \cite{katok:92} that $\D$ (and hence SSPD(2)) is a two-dimensional Riemannian space of constant sectional curvature equal to -1 for the distance noted $d_2$ defined by
\bqs
d_2(z,z')=\text{arctanh}\dfrac{|z-z'|}{|1-\bar{z}z'|}.
\eqs
It follows, e.g. \cite{moakher:05,chossat-faugeras:09}, that SDP(2) is a three-dimensional Riemannian space of constant sectional curvature equal to -1 for the distance noted $d_0$ defined by
\bqs
d_0(\T,\T')=\sqrt{2(\log \Delta-\log \Delta')^2+d_2^{2}(z,z')}
\eqs
As shown in proposition \eqref{prop:volel} of appendix \ref{appendix:volume} it is possible to express the volume element $d\T$ in $(z_1,z_2,\Delta)$ coordinates with $z=z_1+iz_2$:
\bqs
d\T=8\sqrt{2}\dfrac{d\Delta}{\Delta}\dfrac{dz_1dz_2}{(1-|z|^2)^2}
\eqs
We rewrite (\ref{eq:neurmass}) in $(z,\Delta)$ coordinates:
\bqs
\partial_t V(z,\Delta, t)=-\alpha V(z,\Delta, t)+8\sqrt{2}\int_{0}^{+\infty}\int_{\D}W(z,\Delta,z',\Delta',t)S(V(z',\Delta',t))\dfrac{d\Delta'}{\Delta'}\dfrac{dz_1'dz_2'}{(1-|z'|^2)^2}+I(z,\Delta,t)
\eqs
We get rid of the constant $8\sqrt{2}$ by redefining $W$ as $8\sqrt{2}W$.
\bq
\left\{ \begin{array}{lcl}
\partial_t V(z,\Delta, t)&=&-\alpha V(z,\Delta, t)+\int_{0}^{+\infty}\int_{\D}W(z,\Delta,z',\Delta',t)S(V(z',\Delta',t))\dfrac{d\Delta'}{\Delta'}\dfrac{dz_1'dz_2'}{(1-|z'|^2)^2}+I(z,\Delta,t)\\
&&\qquad \qquad \qquad \qquad \qquad \qquad \qquad \qquad \qquad \qquad \qquad \qquad \qquad \qquad \qquad \qquad \qquad\forall t > 0\\
V(z,\Delta,0)&=&V_0(z,\Delta)
\end{array}
\right.
\label{eq:neurco}
\eq
The aim of the following sections is to establish that \eqref{eq:neurco} is well-defined and to give necessary and sufficient conditions on the different parameters in order to prove some results on the existence and uniqueness of a solution of (\ref{eq:neurco}).

\section{The existence and uniqueness of a solution}\label{section:exun}
The aim of this section is to give theoretical and general results of existence and uniqueness of a solution of \eqref{eq:neurmass}.
In the first subsection \ref{subsection:hom} we study the simpler case of homogeneous solutions of (\ref{eq:neurmass}), i.e. of solutions that are constant with respect to the tensor variable $\T$. We then we study in \ref{subsection:semihom} the useful case of the semi-homogeneous solutions of (\ref{eq:neurmass}), i.e. of solutions that depend on the tensor variable but only through its $z$ coordinate in $\D$, and we end up in \ref{subsection:gensol} with the general case.

\subsection{Homogeneous solutions}\label{subsection:hom}

A homogeneous solution to (\ref{eq:neurmass}) is a solution $V$ that does not depend upon the tensor variable $\T$ for a given homogenous input $I(t)$ and a constant initial condition $V_0$. In $(z,\Delta)$ coordinates, a homogeneous solution of (\ref{eq:neurco}) is defined by:
\bqs
\dot V(t)=-\alpha V(t)+\overline{W}(z,\Delta,t)S(V(t))+I(t)
\eqs
where:
\bq
\overline{W}(z,\Delta,t)\overset{def}{=}\int_{0}^{+\infty}\int_{\D}W(z,\Delta,z',\Delta',t)\dfrac{d\Delta'}{\Delta'}\dfrac{dz_1'dz_2'}{(1-|z'|^2)^2}
\label{eq:dblint}
\eq
Hence necessary conditions for the existence of a homogeneous solution are that:
\begin{itemize}
 \item the double integral \eqref{eq:dblint} is convergent,
 \item $\overline{W}(z,\Delta,t)$ does not depend upon the variable $(z,\Delta)$. In that case, we note $\overline{W}(t)$ instead of $\overline{W}(z,\Delta,t)$.
\end{itemize}
 In the special case where $W(z,\Delta,z',\Delta',t)$ is a function of only the distance $d_0$ between $(z,\Delta)$ and $(z',\Delta')$: $$W(z,\Delta,z',\Delta',t)\equiv w\bigg(\sqrt{2(\log \Delta-\log \Delta')^2+d_2^{2}(z,z')},t\bigg)$$ the second condition is satisfied. We postpone the verification of this fact until the following section. To summarize, the homogeneous solutions satisfy the differential equation:
\bq
\left\{ \begin{array}{lcl}
\dot V(t)&=&-\alpha V(t)+\overline{W}(t)S(V(t))+I(t)\quad t > 0\\
V(0)&=&V_0
\end{array}
\right.
\label{eq:homsimp}
\eq

\subsubsection{A first existence and uniqueness result}

Equation (\ref{eq:neurco}) defines a Cauchy problem and we have the following theorem.
\newtheorem{thm}{Theorem}[subsection]
\newtheorem{prop}{Proposition}[subsection]
\begin{thm}
 If the external current $I(t)$ and the connectivity function $\overline{W}(t)$ are continuous on some closed interval $J$ containing $0$, then for all $V_0$ in $\R$, there exists a unique solution of (\ref{eq:homsimp}) defined on a subinterval $J_0$ of $J$ containing $0$ such that $V(0)=V_0$.
\end{thm}
\begin{proof}
 It is a direct application of Cauchy' theorem on differential equations. We consider the mapping $f:J\times \R\rightarrow \R$ defined by:
\bqs
f(t,x)=-\alpha x+\overline{W}(t)S(x)+I(t)
\eqs
It is clear that $f$ is continuous from $J \times \R$ to $\R$. We have for all $x,y\in\R$ and $t\in J$:
\bqs
| f(t,x)-f(t,y)| \leq \alpha | x-y|+|\overline{W}(t)|S'_m |x-y|
\eqs
where $S'_m=\sup_{x\in\R}|S'(x)|$.\\
Since, $\overline{W}$ is continuous on the compact interval $J$, it is bounded there by $C>0$ and:
\bqs
| f(t,x)-f(t,y)| \leq (\alpha +CS'_m) |x-y|
\eqs
\end{proof}
We can extend this result to the whole time real line if $I$ and $\overline{W}$ are continuous on $\R$. 
\begin{prop}
 If $I$ and $\overline{W}$ are continuous on $\R$, then for all $V_0$ in $\R$, there exists a unique solution of (\ref{eq:homsimp}) defined on $\R$ such that $V(0)=V_0$.
\end{prop}
\begin{proof}
 We have already shown the following inequality:
\bqs
| f(t,x)-f(t,y)| \leq \alpha | x-y|+|\overline{W}(t)|S'_m |x-y|
\eqs
Then $f$ is locally Lipschitz with respect to its second argument. Let $V$ be a maximal solution on $J_0$ and we denote by $\beta$ the upper bound of $J_0$. We suppose that $\beta<+\infty$. Then we have for all $t\geq 0$:
\bqs
V(t)= e^{-\alpha t}V_0+\int_0^te^{-\alpha (t-u)}\overline{W}(u)S(V(u))du +\int_0^te^{-\alpha (t-u)}I(u)du
\eqs
\bqs
\Rightarrow |V(t)|\leq |V_0| +S^m\int_0^\beta e^{\alpha u}|\overline{W}(u)|du+\int_0^\beta e^{\alpha u}|I(u)|du \quad \forall t\in[0,\beta]
\eqs
where $S^m=\sup_{x\in\R}|S(x)|$.\\
But theorem \ref{thm:bouts} ensures that it is impossible, then $\beta=+\infty$. The same proof with the lower bound of $J_0$ gives the conclusion.
\end{proof}

\subsubsection{Simplification of (\ref{eq:dblint}) in a special case}\label{subsubsection:simpl}

\paragraph{Invariance}In the previous section, we have stated that in the special case where $W$ was a function of the distance between two points in $\D\times\R_{*}^{+}$, then $\overline{W}(z,\Delta,t)$ did not depend upon the variables $(z,\Delta)$. We now prove this assumption.
\newtheorem{lem}{Lemma}[subsection]
\begin{lem}\label{lem:invariance}
When $W$ is only a function of $d_0(\T,\T')$, then $\overline{W}$ does not depend upon the variable $\T$.
\end{lem}
\begin{proof}
We work in $(z,\Delta)$ coordinates and we begin by rewritting the double integral \eqref{eq:dblint} for all $(z,\Delta)\in\R_{*}^+\times\D$:
\bqs
\overline{W}(z,\Delta,t)=\int_{0}^{+\infty}\int_{\D}W\bigg(\sqrt{2(\log \Delta-\log \Delta')^2+d_2^{2}(z,z')},t\bigg)\dfrac{d\Delta'}{\Delta'}\dfrac{dz_1'dz_2'}{(1-|z'|^2)^2}
\eqs
The change of variable $\Delta'\rightarrow \frac{\Delta'}{\Delta}$ yields:
\bqs
\overline{W}(z,\Delta,t)=\int_{0}^{+\infty}\int_{\D}W\bigg(\sqrt{2(\log \Delta')^2+d_2^{2}(z,z')},t\bigg)\dfrac{d\Delta'}{\Delta'}\dfrac{dz_1'dz_2'}{(1-|z'|^2)^2}
\eqs
And it establishes that $\overline{W}$ does not depend upon the variable $\Delta$. To finish the proof, we show that the following integral does not depend upon the variable $z\in\D$:
\bq
\Xi(z)=\int_{\D}f(d_2(z,z'))\dfrac{dz_1'dz_2'}{(1-|z'|^2)^2}
\label{eq:Xi}
\eq
where $f$ is a real-valued function such that $\Xi(z)$ is well defined.\\
We express $z$ in horocyclic coordinates: $z=n_s a_r .O$ (see appendix \ref{section:isom}) and (\ref{eq:Xi}) becomes:
\bqs
\Xi(z)=\int_{\R}\int_{\R}f(d_2(n_{s}a_{r}.O,n_{s'}a_{r'}.O))e^{-2r'}ds'dr'
\eqs
\bqs
=\int_{\R}\int_{\R}f(d_2(n_{s-s'}a_{r}.O,a_{r'}.O))e^{-2r'}ds'dr'
\eqs
With the change of variable $s-s'=-xe^{2r}$, this becomes:
\bqs
\Xi(z)=\int_{\R}\int_{\R}f(d_2(n_{-xe^{2r}}a_{r}.O,a_{r'}.O))e^{-2(r'-r)}dxdr'
\eqs
The relation $n_{-xe^{2r}}a_{r}.O=a_{r}n_{-x}.O$ (proved e.g. in \cite{helgason:00}) yields:
\bqs
\Xi(z)=\int_{\R}\int_{\R}f(d_2(a_{r}n_{-x}.O,a_{r'}.O))e^{-2(r-r)'}dxdr'
\eqs
\bqs
=\int_{\R}\int_{\R}f(d_2(O,n_{x}a_{r'-r}.O))e^{-2(r'-r)}dxdr'
\eqs
\bqs
=\int_{\R}\int_{\R}f(d_2(O,n_{x}a_{y}.O))e^{-2y}dxdy
\eqs
\bqs
=\int_{\D}f(d_2(O,z'))\text{dm}(z')
\eqs
with $z'=z_1'+i z_2'$ and $\text{dm}(z')=\dfrac{dz_1'dz_2'}{(1-|z'|^2)^2}$, which shows that $\Xi(z)$ does not depend upon the variable $z$, as announced.  
\end{proof}

\paragraph{Mexican hat connectivity} In this paragraph, we push further the computation of $\overline{W}$ in the special case where $W$ does not depend upon the time variable $t$ and takes the special form suggested by Amari in \cite{amari:77}, commonly referred to as the “Mexican hat” connectivity. It features center excitation and surround inhibition which is an effective model for a mixed population of interacting inhibitory and excitatory neurons with typical cortical connections. It is also only a function of $d_0(\T,\T')$.\\
In detail, we have:
\bqs
W(z,\Delta,z'\Delta')=w\bigg(\sqrt{2(\log \Delta-\log \Delta')^2+d_2^{2}(z,z')}\bigg)
\eqs
where:
\bqs
w(x)=\dfrac{1}{\sqrt{2\pi\sigma_1^2}}e^{-\dfrac{x^2}{\sigma_1^2}}-
\dfrac{A}{\sqrt{2\pi\sigma_2^2}}e^{-\dfrac{x^2}{\sigma_2^2}}
\eqs
with $0\leq \sigma_1\leq\sigma_2$ and $0\leq A\leq 1$.\\
In this case we can obtain a very simple closed-form formula for $\overline{W}$ as shown in the following lemma.
\begin{lem}\label{lemma:mexican}
 When $W$ is a mexican hat function of $d_2(\T,\T')$ and independent of $t$, then:
\bq
\overline{W}=\frac{\pi^{\frac{3}{2}}}{2}\left(\sigma_1e^{2\sigma_1^2}\textbf{erf}\left(\sqrt{2}\sigma_1\right)-A\sigma_2e^{2\sigma_2^2}\textbf{erf}\left(\sqrt{2}\sigma_2\right) \right)
\label{eq:mexhat}
\eq
where $\textbf{erf}$ is the error function defined as:
\bqs
\textbf{erf}(x)=\frac{2}{\sqrt{\pi}}\int_0^x e^{-u^2}du
\eqs
\end{lem}
\begin{proof}
The proof is given in appendix \ref{appendix:mexican}.
\end{proof}

\subsection{Semi-homogeneous solutions}\label{subsection:semihom}

A semi-homogeneous solution of (\ref{eq:neurco}) is defined as a solution which does not depend upon the variable $\Delta$. In other words, the populations of neurons are not sensitive to the determinant of the structure tensor. The neural mass equation is then equivalent to the neural mass equation for tensors of unit determinant. We point out that semi-homogeneous solutions were previously introduced by Chossat and Faugeras in \cite{chossat-faugeras:09}. They also performed a bifurcation analysis of what they called H-planforms. In this section, we define the framework in which their equations make sense.
\bq
\left\{
\begin{array}{lcl}
\partial_t V(z,t)&=&-\alpha V(z,t)+\int_{\D} W^{sh}(z,z',t)S(V(z',t))\text{dm}(z')+I(z,t) \quad t > 0\\
V(z,0) &=& V_0(z)
\end{array}
\right.
\label{eq:semihom}
\eq

where
\bqs
W^{sh}(z,z',t)=\int_{0}^{+\infty}W(z,\Delta,z',\Delta',t)\dfrac{d\Delta'}{\Delta'} 
\eqs
We have implicitly made the assumption, that $W^{sh}$ does not depend on the coordinate $\Delta$. Some conditions under which this assumption is satisfied are described below.

We now deal with the problem of the existence and uniqueness of a solution to (\ref{eq:semihom}) for a given initial condition. We  first introduce the framework in which this equation makes sense.

\subsubsection{The well-posedness of equation (\ref{eq:semihom})}

Let $J$ be an open interval of $\R$. We assume that:
\begin{itemize}
 \item (\textbf{C1}): $\forall(z,z',t)\in\D^2\times J$, $W^{sh}(z,z',t)\equiv w^{sh}(d_2(z,z'),t)$,
 \item (\textbf{C2}): $\mbW^{sh}\in\mC \left(J,L^1(\D)\right)$ where $\mbW^{sh}$ is defined as $\mbW^{sh}(z,t)=w^{sh}(d_2(z,0),t)$ for all\\ $(z,t)\in\D\times J$, 
 \item (\textbf{C3}): $\sup_{t\in J}\| \mbW^{sh}(t)\|_{L^1}<+\infty$ where $\|\mbW^{sh}(t)\|_{L^1} \overset{def}{=}\int_{\D} |W^{sh}(d_2(z,0),t)|\text{dm}(z)$.
\end{itemize}
Note that conditions (\textbf{C1})-(\textbf{C2}) and lemma \ref{lem:invariance} imply that for all $z\in\D$, $\int_{\D}\left|W^{sh}(z,z',t)\right|\text{dm}(z')=\|\mbW^{sh}(t)\|_{L^1}$. And then, for all $z\in\D$, the mapping $z'\rightarrow W^{sh}(z,z',t)$ is integrable on $\D$. We deduce that for all $(z,t)\in \D\times J$:
\bqs
\left|\int_{\D} W^{sh}(z,z',t)S(V(z',t))\text{dm}(z')\right|\leq \int_{\D} \left|W^{sh}(z,z',t)\right|\left|S(V(z',t))\right|\text{dm}(z')
\eqs
\bqs
\leq S^m \int_{\D} \left|W^{sh}(z,z',t)\right|\text{dm}(z')\leq S^m \sup_{t\in J}\|\mbW^{sh}(t)\|_{L^1}<+\infty
\eqs
The last inequality is a consequence of lemma \ref{lem:invariance} which shows that $\int_{\D} \left|W^{sh}(z,z',t)\right|\text{dm}(z')$ is not a function of $z$.

Finally for all $z\in\D$ and $t\in J$, the righthand side of equation (\ref{eq:semihom}) is well-defined.\\
We introduce the following mapping, defined on $J \times \mF$, where $\mF$ is a yet to be defined functional space:
\bq
f:(t,\phi) \rightarrow f(t,\phi) \text{ such that } f(t,\phi)(z)=\int_{\D} W^{sh}(z,z',t)S(\phi(z'))\text{dm}(z')
\label{eq:mapsh}
\eq
Our aim is to find a functional space $\mF$ where (\ref{eq:mapsh}) is well-defined and the function $f$ maps $\mF$ to $\mF$ for all $t$s. \\
A natural choice would be to choose $\phi$ as a  $L^p(\D,dm)$-integrable function of the space variable with $1\leq p<+\infty$. Unfortunately, the homogeneous solutions (constant with respect to $z$) do not belong to that space. Moreover, a valid model of neural networks should only produce bounded membrane potentials. That is why we focus our choice on functional spaces of bounded, or essentially bounded, functions.

\paragraph{The well-posedness of equation (\ref{eq:semihom}) in the case $\mF=L^{\infty}(\D)$}

Our first choice is $\mF=L^{\infty}(\D)$. The Fischer-Riesz's theorem ensures that $L^{\infty}(\D)$ is a Banach space for the norm: $\| \psi \|_{\infty}=\inf \{C\geq 0,\text{ }|\psi(z)|\leq C \text{ for almost every } z\in\D \}$. We have the following proposition.

\begin{prop}
 If $W^{sh}$ satisfies conditions (\textbf{C1})-(\textbf{C3}) then $f$ is well defined and is from $J \times L^{\infty}(\D)$ to $L^{\infty}(\D)$.
\end{prop}
\begin{proof}
 For all $\phi\in\mF$ and $t\in\R$, $f(t,\phi)$ is obviously measurable and we have from (\ref{eq:mapsh}):
\bqs
\| f(t,\phi) \|_{\infty}\leq S^m \sup_{t\in J}\|\mbW^{sh}(t)\|_{L^1}<+\infty
\eqs
and hence $f(t,\phi )\in L^{\infty}(\D)$.
\end{proof}

\paragraph{The well-posedness of equation (\ref{eq:semihom}) in the case $\mF=\mC_b(\D)$}

We now choose $\mF=\mC_b(\D)$ the space of the bounded continuous functions on $\D$. As shown in, e.g. \cite{hewitt-stromberg:65}, this is a Banach space  with respect to the uniform norm: $\| \psi \|_{\mC_b(\D)}=\sup_{z\in\D}|\psi(z)|$. The previous proposition \ref{prop:wellp} still holds:
\begin{prop}
 If $W^{sh}$ satisfies conditions (\textbf{C1})-(\textbf{C3}) then $f$ is well defined and is from $J\times \mC_b(\D)$ to $\mC_b(\D)$.
\end{prop}
\begin{proof}
 For all $\phi \in \mC_b(\D)$ and $t\in\R$:
\bqs
\| f(t,\phi) \|_{\mC_b(\D)}\leq S^m \sup_{t\in J}\|\mbW^{sh}(t)\|_{L^1}<+\infty
\eqs
so $f(t,\phi)$ is bounded. It remains to show that for all $(t,z,\phi)\in J\times\D\times\mC_b(\D)$ the mapping $z\rightarrow f(t,\phi)(z)$ is continuous on $\D$. We fix $t\in J$ and write $z$ in horocyclic coordinates $z=n_{s}a_{r}.O$ with $(s,r)\in\R^2$ and define $\tilde{f}$ on $\R^2$ as:
\bqs
\tilde{f}:(s,r)\rightarrow \int_{\R}\int_{\R}W^{sh}\big(d_2(n_{s}a_{r}.O,n_{s'}a_{r'}.O),t\big)S(\phi(n_{s'}a_{r'}.O))e^{-2r'}ds'dr'
\eqs
A method similar to the one used in section \ref{subsubsection:simpl} leads to:
\bqs
\tilde{f}(s,r)=\int_{\R}\int_{\R}W^{sh}\big(d_2(O,n_{x}a_{y}.O),t\big)S(\phi(n_{s+xe^{2r}}a_{r+y}.O))e^{-2y}dxdy
\eqs
We now use a theorem on integrals depending on a parameter. It is easy to verify that
\begin{enumerate}
 \item for all $(s,r)\in\R^2$ the function $(x,y)\rightarrow W^{sh}\big(d_2(O,n_{x}a_{y}.O),t\big)S(\phi(n_{s+xe^{2r}}a_{r+y}.O))e^{-2y}$ is measurable on $\R^2$,
 \item for almost every $(x,y)\in\R^2$ the function $(s,r)\rightarrow W^{sh}\big(d_2(O,n_{x}a_{y}.O),t\big)S(\phi(n_{s+xe^{2r}}a_{r+y}.O))e^{-2y}$ is continuous on $\R^2$,
 \item for all $(s,r)\in\R^2$, 
\bqs
|W^{sh}\big(d_2(O,n_{x}a_{y}.O),t\big)S(\phi(n_{s+xe^{2r}}a_{r+y}.O))e^{-2y}|\leq |W^{sh}\big(d_2(O,n_{x}a_{y}.O),t\big)|S^m
\eqs
and $(x,y)\rightarrow|W^{sh}\big(d_2(O,n_{x}a_{y}.O),t\big)|$ is integrable on $\R^2$.
\end{enumerate}
It follows that the function $\tilde{f}$ is continuous on $\R^2$ and $f$ is continuous on $\D$.\\
Finally, $f(t,\phi )$ belongs to $\mC_b(\D)$.
\end{proof}

\subsubsection{The existence and uniqueness of a solution of (\ref{eq:semihom})}
From now on, $\mF$ is a functional Banach space for the norm $\| \cdot \|_{\mF}$. We suppose that all the hypotheses are verified so that $f$ is well-defined from $J\times\mF$ to $\mF$ with $J$ an open interval containing $0$. In the previous section we have already presented two different examples for $\mF$: $L^{\infty}(\D)$ and $\mC_b(\D)$.\\
We rewrite (\ref{eq:semihom}) as a Cauchy problem defined on $\mF$:

\bq \left\{
    \begin{array}{lcl}
V'(t)&=&-\alpha V(t)+f(t,V(t))+I(t)\quad t\in J\\
V(0)&=&V_0
\end{array}
\right.
\label{eq:shcauchy}
\eq
\begin{thm}\label{thm:exun}
 If the external current $I$ belongs to $\mC(J,\mF)$ with $J$ an open interval containing $0$ and $W^{sh}$ satisfies conditions (\textbf{C1})-(\textbf{C3}), then for all $V_0\in\mF$, there exists a unique solution of (\ref{eq:shcauchy}) defined on a subinterval $J_0$ of $J$ containing $0$.
\end{thm}

\begin{proof}
 We prove that $f$ is continuous on $J\times\mF$. We have
\begin{multline*}
f(t,x)-f(s,y)=\int_{\D} W^{sh}(\cdot,z',t)\big(S(x)-S(y)\big)\text{dm}(z')\\
+\int_{\D}\bigg( W^{sh}(\cdot,z',t)-W^{sh}(\cdot,z',s)\bigg)S(y)\text{dm}(z'),
\end{multline*}
and therefore
\bqs
\| f(t,x)-f(s,y) \|_{\mF}\leq  S'_m \mW_0\|x-y\|_{\mF}
+S^{m}\left\| \mbW^{sh}(t)-\mbW^{sh}(s)\right\|_{L^1}
\eqs
Because of condition (\textbf{C2}) we can choose $|t-s|$ small enough so that\\ $\left\| \mbW^{sh}(t)-\mbW^{sh}(s)\right\|_{L^1}$ is arbitrarily small. This proves the continuity of $f$. Moreover it follows from the previous inequality that:
\bqs
\| f(t,x)-f(t,y) \|_{\mF}\leq  S'_m \mW_0\|x-y\|_{\mF} 
\eqs
with $\mW_0=\sup_{t\in J}\|\mbW^{sh}(t)\|_{L^1}$. This ensures the Lipschitz continuity of $f$ with respect to its second argument, uniformly with respect to the first. The Cauchy-Lipschitz theorem on a Banach space yields the conclusion. 
\end{proof}
This solution, defined on the subinterval $J$ of $\R$ can in fact be extended to the whole real line, and we have the following proposition. 

\begin{prop}\label{prop:exunext}
 If the external current $I$ belongs to $\mC(\R,\mF)$ and $W^{sh}$ satisfies conditions (\textbf{C1})-(\textbf{C3}) with $J=\R$, then for all $V_0\in\mF$, there exists a unique solution of (\ref{eq:shcauchy}) defined on $\R$.
\end{prop}

\begin{proof}
 The proof is a direct application of theorem \ref{thm:appendixglobal}. If $f$ is Lipschitz continuous with respect to its second argument, the righthand side of \eqref{eq:shcauchy} is also  Lipschitz continuous. Now $f$ is Lipschitz continuous because
\bqs
\| f(t,x)-f(t,y) \|_{\mF}\leq  S'_m \mW_0\|x-y\|_{\mF} \quad \forall t\in \R
\eqs
\end{proof}

\subsubsection{The boundedness of a solution of (\ref{eq:semihom})}\label{subsubsection:intr}
We assume that $\mF$ is a Banach space chosen so that the mapping $f$ is well defined from $\R\times\mF$ to $\mF$. Then the following propostion holds.

\begin{prop}
 If the external current $I$ belongs to $\mC(\R^+,\mF)$ and is bounded in time, i.e. $\sup_{t\in\R^+}\| I(t)\|_{\mF}<+\infty$, and $W^{sh}$ satisfies conditions (\textbf{C1})-(\textbf{C3}) with $J=\R^{+*}$, then the solution of (\ref{eq:shcauchy}) is bounded for each initial condition $V_0\in\mF$.
\end{prop}

\begin{proof}
For all $t\in\R^{+*}$ we integrate (\ref{eq:semihom}) over $[0,t]$:
\bqs
V(z,t)=e^{-\alpha t} V_0(z)+\int_0^t e^{-\alpha(t-u)}\int_{\D} W^{sh}(z,z',u)S(V(z',u))\text{dm}(z')du+\int_0^t e^{-\alpha(t-u)}I(z,u)du
\eqs
The following upperbound holds
\bq
\|V(t)\|_{\mF}\leq e^{-\alpha t} \|V_0\|_{\mF}+\frac{1}{\alpha}\left(S^m \mW_0+\sup_{t\in\R^+}\| I(t)\|_{\mF}\right)\left(1-e^{-\alpha t}\right),
\label{eq:majoration}
\eq
and hence
\bqs
\|V(t)\|_{\mF}\leq \|V_0\|_{\mF}+\frac{1}{\alpha}\left(S^m \mW_0+\sup_{t\in\R^+}\| I(t)\|_{\mF}\right),
\eqs
which shows that the solution is bounded for each initial condition $V_0\in\mF$.\\
\end{proof}
The upperbound \eqref{eq:majoration} yields a simple attracting set for the dynamics of (\ref{eq:semihom}) as shown in the following proposition.
\begin{prop}
 Let $\rho\overset{def}{=}\frac{2}{\alpha}\left(S^m \mW_0+\sup_{t\in\R^+}\| I(t)\|_{\mF}\right)$. The open ball $B_{\rho}$ of $\mF$ of center $0$ and radius $\rho$ is stable under the dynamics  of equation (\ref{eq:semihom}). Moreover it is an attracting
set for this dynamics and if $V_0\notin B_{\rho}$ and $T=\inf \{t>0 \text{ such that } V(t)\in B_{\rho}\}$ then:
\bqs
T\leq \frac{1}{\alpha}\log\left(\frac{2\| V_0\|_{\mF}-\rho}{\rho} \right)
\eqs
\end{prop}
\begin{proof}
 We can rewrite \eqref{eq:majoration} as:
 \begin{multline}\label{eq:upperbound}
\| V(t)\|_{\mF} \leq e^{-\alpha t}\left( \|V_0\|_{\mF}-\frac{1}{\alpha}\left(S^m \mW_0+\sup_{t\in\R^+}\| I(t)\|_{\mF}\right)\right) +\frac{1}{\alpha}\left(S^m \mW_0+\sup_{t\in\R^+}\| I(t)\|_{\mF}\right) =\\
e^{-\alpha t}\left( \|V_0\|_{\mF}-\frac{\rho}{2}\right)+\frac{\rho}{2}
\end{multline}
If $V_0 \in B_\rho$ this implies $\| V(t)\|_{\mF} \leq \frac{\rho}{2}\left(1+e^{-\alpha t} \right)$ for all $t > 0$ and hence $\| V(t)\|_{\mF} < \rho$ for all $t>0$, proving that $B_\rho$ is stable.
Now assume that $\| V(t)\|_{\mF} > \rho$ for all $t \geq 0$. The inequality \eqref{eq:upperbound} shows that for $t$ large enough this yields a contradiction. Therefore there exists $t_0 > 0$ such that $\| V(t_0)\|_{\mF}=\rho$. At this time instant
we have
\bqs
\rho\leq e^{-\alpha t_0}\left( \|V_0\|_{\mF}-\frac{\rho}{2}\right) +\frac{\rho}{2},
\eqs
and hence
\bqs
t_0\leq \frac{1}{\alpha}\log\left(\frac{2\| V_0\|_{\mF}-\rho}{\rho} \right)
\eqs
\end{proof}

\newtheorem{cor}{Corollary}[subsection]



\subsection{General solution}\label{subsection:gensol}

We now deal with the general solutions of equation (\ref{eq:neurmass}).  We first give some hypotheses that the connectivity function $W$ must satisfy. We present them in two ways, first on the set of structure tensors considered as the set SPD(2) and, second on the set of tensors seen as $D\times \R_*^+$. Let $J$ be a subinterval of $\R$. We assume that:
\begin{itemize}
 \item (\textbf{H1}): $\forall (\T,\T',t)\in \text{SPD(2)}\times \text{SPD(2)}\times J$, $W(\T,\T',t)\equiv W(d_0(\T,\T'),t)$,
 \item (\textbf{H2}): $\mbW\in\mC\left(J,L^1\left(\text{SPD(2)}\right)\right)$ where $\mbW$ is defined as $\mbW(\T,t)=W\left(d_0(\T,\text{Id}_2),t\right)$ for all $(\T,t)\in \text{SPD(2)}\times J$ where $\text{Id}_2$ is the identity matrix of $\mM_2(\R)$,
 \item (\textbf{H3}): $\forall t\in J$, $\sup_{t\in J}\| \mbW(t)\|_{L^1}<+\infty$ where $\| \mbW(t)\|_{L^1}\overset{def}{=}\int_{\text{SPD(2)}} |W(d_0(\T,\text{Id}_2),t)|d\T$.
\end{itemize}
We now express these hypotheses for the representation in $(z,\Delta)\in\D\times\R_{*}^{+}$ of structure tensors:
\begin{itemize}
 \item (\textbf{H1}bis): $\forall (z,z',\Delta,\Delta',t)\in \D^2\times (\R_{*}^{+})^2\times \R$, $W(z,\Delta,z',\Delta',t)\equiv W(d_2(z,z'),|\log(\Delta)-\log(\Delta')|,t)$,
 \item (\textbf{H2}bis): $\mbW\in\mC\left(J,L^1\left(\D\times \R_{*}^{+}\right)\right)$ where $\mbW$ is defined as $\mbW(z,\Delta,t)=W\left(d_2(z,0),|\log(\Delta)|,t\right)$ for all $\forall(z,\Delta,t)\in\D\times \R_{*}^{+}\times J$,
 \item (\textbf{H3}bis): $\forall t\in J$, $\sup_{t\in J}\|\mbW(t)\|_{L^1}<+\infty$ where 
\bqs 
\|\mbW(t)\|_{L^1}\overset{def}{=}\int_{\D\times \R_{*}^{+}}\left|W\left(d_2(z,0),\left|\log(\Delta)\right|,t\right)\right|\frac{d\Delta}{\Delta}dm(z)
\eqs.
\end{itemize}

\subsubsection{Functional space setting}
We need to settle on the choice of a Banach functional space $\mF$ for the membrane potential as in section \ref{subsection:semihom}.
Our study of the semi-homogeneous case suggests the following choice: $\mF=L^{\infty}(\D\times\R_{*}^{+})$. As $\D\times\R_{*}^{+}$ is an open set of $\R^3$, $\mF$ is a Banach space for the norm: $\| \phi \|_{\mF}=\sup_{z\in\D}\sup_{\Delta\in\R_{*}^{+}}|\phi(z,\Delta)|$.\\
We introduce the following mapping $f^{g}:(t,\phi) \rightarrow f^{g}(t,\phi)$ such that:
\bq
 f^{g}(t,\phi)(z,\Delta)=-\alpha \phi(z,\Delta)+\int_{\D\times\R_{*}^{+}}W\left(d_2(z,z'),\left|\log(\frac{\Delta}{\Delta'})\right|,t\right) S\left(\phi(z',\Delta')\right)\dfrac{d\Delta'}{\Delta'}\text{dm}(z')+I(z,\Delta,t)
\label{eq:map}
\eq
\begin{prop}\label{prop:wellp}
 If $I\in\mC\big(J,\mF\big)$ with $\sup_{t\in J}\| I(t)\|_{\mF}<+\infty$ and $W$ satisfies hypotheses (\textbf{H1}bis)-(\textbf{H3}bis) then $f^{g}$ is well-defined and is from $J\times\mF$ to $\mF$.
\end{prop}
\begin{proof}
 $\forall(z,\Delta,t)\in\D\times \R_{*}^{+}\times\R$, we have:
\bqs
\left|\int_{\D\times\R_{*}^{+}}W\left(d_2(z,z'),\left|\log(\frac{\Delta}{\Delta'})\right|,t\right) S(\phi(z',\Delta'))\dfrac{d\Delta'}{\Delta'}\text{dm}(z')\right|\leq S^{m}\sup_{t\in J}\|\mbW(t)\|_{L^1}<+\infty
\eqs
\end{proof}

\subsubsection{The existence and uniqueness of a solution of (\ref{eq:neurco})}
We rewrite (\ref{eq:neurco}) as a Cauchy problem:

\bq \left\{
    \begin{array}{ll}
\partial_t V(z,\Delta,t)=-\alpha V(z,\Delta, t)+\int_{\D\times\R_{*}^{+} }W\left(d_2(z,z'),\left|\log(\frac{\Delta}{\Delta'})\right|,t\right)S(V(z',\Delta',t))\dfrac{d\Delta'}{\Delta'}\text{dm}(z')+I(z,\Delta,t)\\
V(z,\Delta,0)=V_0(z,\Delta)
\end{array}
\right.
\label{eq:gencauchy}
\eq

\begin{thm}
 If the external current $I$ belongs to $\mC(J,\mF)$ with $J$ an open interval containing $0$ and $W$ satisfies hypotheses (\textbf{H1}bis)-(\textbf{H3}bis), then fo all $V_0\in\mF$, there exists a unique solution of (\ref{eq:gencauchy}) defined on a subinterval $J_0$ of $J$ containing $0$ such that $V(z,\Delta,0)=V_0(z,\Delta)$ for all $(z,\Delta)\in\D\times\R_{*}^{+}$.
\end{thm}
\begin{proof}
 The proof is a direct adaptation of the proof of theorem \ref{thm:exun}.
\end{proof}
\begin{prop}
 If the external current $I$ belongs to $\mC(\R,\mF)$ and $W$ satisfies hypotheses (\textbf{H1}bis)-(\textbf{H3}bis) with $J=\R$, then for all $V_0\in\mF$, there exists a unique solution of (\ref{eq:gencauchy}) defined on $\R$ such that $V(z,\Delta,0)=V_0(z,\Delta)$ for all $(z,\Delta)\in\D\times\R_{*}^{+}$.
\end{prop}
\begin{proof}
The proof is readily adapted from that of proposition \ref{prop:exunext}.
\end{proof}
\subsubsection{The intrinsic boundedness of a solution of (\ref{eq:neurco})}

In the same way as in the homogeneous case, we show a result on the boundedness of a solution of (\ref{eq:neurco}). The proofs of the following properties are exactely the same as those in section \ref{subsubsection:intr}.
\begin{prop}
 If the external current $I$ belongs to $\mC(\R^+,\mF)$ and is bounded in time $\sup_{t\in\R^+}\| I(t)\|_{\mF}<+\infty$ and $W$ satisfies hypotheses (\textbf{H1}bis)-(\textbf{H3}bis) with $J=\R^{+}$, then the solution of (\ref{eq:gencauchy}) is bounded for each initial condition $V_0\in\mF$.
\end{prop}
If we set:
\bqs
\rho^{g}\overset{def}{=}\dfrac{2}{\alpha}(S^m \mW_0^{g}+\sup_{t\in\R^+}\| I(t)\|_{\mF})
\eqs
where $\mW_0^{g}=\sup_{t\in \R^{+}}\|\mbW(t)\|_{L^1}$. The following corollary is a consequence of the previous proposition.
\begin{cor}
 If $V_0\notin B_{\rho^{g}}$ and $T^{g}=\inf \{t>0 \text{ such that } V(t)\in B_{\rho^{g}}\}$ then:
\bqs
T^{g}\leq \frac{1}{\alpha}\log\left(\frac{2\| V_0\|_{\mF}-\rho^g}{\rho^g} \right)
\eqs
\end{cor}

\section{Stationary solutions}
\label{section:stationary}

We look at the equilibrium states, noted $V_{\mu}^{0}$ of (\ref{eq:neurco}), when the external input $I$ and the connectivity $W$ do not depend upon the time. We assume that $W$ satisfies hypotheses (\textbf{H1}bis)-(\textbf{H2}bis). We redefine for convenience the sigmoidal function to be:
\bqs
S(x)=\dfrac{1}{1+e^{-x}},
\eqs
so that a stationary solution (independent of time) satisfies:
\bq
0=-\alpha V_{\mu}^0 (z,\Delta)+\int_{\D}W\left(d_2(z,z'),\left|\log(\frac{\Delta}{\Delta'})\right|\right)S(\mu V_{\mu}^0 (z',\Delta'))\dfrac{d\Delta'}{\Delta'}dm(z')+I(z,\Delta)
\label{eq:genstate}
\eq
We define the nonlinear operator from $\mF$ to $\mF$, noted $\mG_{\mu}$, by:
\bq
\mG_{\mu}(V)(z,\Delta)=\int_{\D\times\R_{*}^{+}}W\left(d_2(z,z'),\left|\log(\frac{\Delta}{\Delta'})\right|\right)S(\mu V(z',\Delta'))\dfrac{d\Delta'}{\Delta'}dm(z')
\label{eq:operator}
\eq
Finally, \eqref{eq:genstate} is equivalent to:
\bqs
\alpha V_{\mu}^0 (z,\Delta)=\mG_{\mu}(V)(z,\Delta)+I(z,\Delta)
\eqs

\subsection{Study of the nonlinear operator $\mG_{\mu}$}

We recall that we have set for the Banach space $\mF=L^{\infty}(\D\times\R_{*}^{+})$ and proposition \ref{prop:wellp} shows that $\mG_{\mu}:\mF\rightarrow\mF$. We have the further properties:
\begin{prop}\label{prop:opG}
 $\mG_{\mu}$ satisfies the following properties:
\begin{itemize}
 \item $\|\mG_{\mu}(V_1)-\mG_{\mu}(V_2) \|_{\mF}\leq \mu W_0^{g}S'_m\|V_1-V_2 \|_{\mF}$ for all $\mu\geq 0$,
 \item $\mu\rightarrow\mG_{\mu}$ is continuous on $\R^+$, 
\end{itemize}
\end{prop}

\begin{proof}
The first property was shown to be true in the proof of theorem \ref{thm:exun}. The second property follows from the following inequality:
\bqs
\|\mG_{\mu_1}(V)-\mG_{\mu_2}(V) \|_{\mF}\leq |\mu_1-\mu_2| W_0^{g}S'_m\|V\|_{\mF}
\eqs
\end{proof}

\noindent
We denote by $\mG_{l}$ and $\mG_{\infty}$ the two operators from $\mF$ to $\mF$ defined as follows for all $V \in \mF$ and all $(z,\Delta)\in\D\times\R_{*}^{+}$:
\bq
\mG_l(V)(z,\Delta)=\int_{\D\times\R_{*}^{+}}W\left(d_2(z,z'),\left|\log(\frac{\Delta}{\Delta'})\right|\right) V(z',\Delta')\dfrac{d\Delta'}{\Delta'}dm(z'),
\label{eq:operatorlinear}
\eq
and 
\bqs
\mG_{\infty}(V)(z,\Delta)=\int_{\D\times\R_{*}^{+}}W\left(d_2(z,z'),\left|\log(\frac{\Delta}{\Delta'})\right|\right)H(V(z',\Delta'))\frac{d\Delta'}{\Delta'}dm(z')
\eqs
where $H$ is the Heaviside function.

It is straightforward to show that both operators are well-defined on $\mF$ and map $\mF$ to $\mF$. Moreover the following proposition holds.
\begin{prop}
We have
 \bqs
\mG_{\mu}\underset{\mu \rightarrow \infty}{\longrightarrow}\mG_{\infty}
\eqs
\end{prop}
\begin{proof}
 It is a direct application of the dominated convergence theorem using the fact that:
\bqs
S(\mu y)\underset{\mu \rightarrow \infty}{\longrightarrow}H(y)\quad \text{ p.p }y\in\R
\eqs
\end{proof}

\subsection{The convolution form of the operator $\mG_{\mu}$ in the semi-homogeneous case} 
It is convenient to consider the functional space $\mF^{sh}=L^{\infty}(\D)$ to discuss semi-homogeneous solutions.
A semi-homogeneous persistent state of (\ref{eq:neurco}) is deduced from (\ref{eq:genstate}) and satisfies:
\bq
\alpha V_{\mu}^0 (z)=\mG_{\mu}^{sh}(V_{\mu}^0 )(z)+I(z)
\label{eq:semhomstate}
\eq
where the nonlinear operator $\mG_{\mu}^{sh}$ from $\mF^{sh}$ to $\mF^{sh}$ is defined for all $V \in \mF^{sh}$ and  $z\in\D$ by:
\bqs
\mG_{\mu}^{sh}(V)(z)=\int_{\D}W^{sh}(d_2(z,z'))S(\mu V(z'))dm(z')
\eqs
We define the associated operators, $\mG_{l}^{sh},\mG_{\infty}^{sh}$:
\bqs
\mG_{l}^{sh}(V)(z)=\int_{\D}W^{sh}(d_2(z,z'))V(z')dm(z')
\eqs
\bqs
\mG_{\infty}^{sh}(V)(z)=\int_{\D}W^{sh}(d_2(z,z'))H(V(z'))dm(z')
\eqs
We rewrite the operator $\mG_{\mu}^{sh}$ in a convenient form by using the convolution in the hyperbolic disk. First, we define the convolution in a such space. Let $dg$ denote the Haar measure on the group $G=SU(1,1)$ (see \cite{helgason:00}), normalized by:
\bqs
\int_G f(g\cdot O)dg\overset{def}{=}\int_{\D}f(z)dm(z),
\eqs
for all functions of $L^1(\D)$.
Given two functions $f_1,f_2$ in $L^1(\D)$ we define the convolution $\ast$ by:
\bqs
(f_1\ast f_2)(z)=\int_{G}f_1(g\cdot O)f_2(g^{-1}\cdot z)dg
\eqs
We recall the notation $\textbf{W}^{sh}(z)\overset{def}{=}W^{sh}(d_2(z,O))$.

\begin{prop}\label{prop:convolution}
 For all $\mu\geq0$ and $V\in\mF^{sh}$ we have:
\bq
\mG_{\mu}^{sh}(V)=\textbf{W}^{sh}\ast S(\mu V),\quad \mG_{l}^{sh}(V)=\textbf{W}^{sh}\ast V\ \text{ and }\ \mG_{\infty}^{sh}(V)=\textbf{W}^{sh}\ast H(V)
\eq
\end{prop}

\begin{proof}
We only prove the result for $\mG_\mu$.
 Let $z\in\D$, then:
\begin{multline*}
\mG_{\mu}^{sh}(V)(z)=\int_{\D}W^{sh}(d_2(z,z'))S(\mu V(z'))dm(z')=\int_{G}W^{sh}(d_2(z,g\cdot O))S(\mu V(g\cdot O))dg\\
=\int_{G}W^{sh}(d_2(gg^{-1}\cdot z,g\cdot O))S(\mu V(g\cdot O))dg
\end{multline*}
and for all $g\in SU(1,1)$, $d_2(z,z')=d_2(g\cdot z,g\cdot z')$ so that:
\bqs
\mG_{\mu}^{sh}(V)(z)=\int_{G}W^{sh}(d_2(g^{-1}\cdot z,O))S(\mu V(g\cdot O))dg=\textbf{W}^{sh}\ast S(\mu V)(z)
\eqs
\end{proof}
Let $b$ be a point on the circle $\partial \D$. For $z\in\D$, we define the ``inner product'' $<z,b>$ to be the algebraic distance to the origin of the (unique) horocycle based at $b$ through $z$ (see \cite{chossat-faugeras:09}). Note that $<z,b>$ does not depend on the position of $z$ on the horocycle.
The Fourier transform in $\D$ is defined as (see \cite{helgason:00}):
\bqs
\widetilde{h}(\lambda,b)=\int_{\D}h(z)e^{(-i\lambda +1)<z,b>}dm(z)\quad \forall (\lambda,b)\in\R\times\partial\D
\eqs
for a function $h:\D\rightarrow \mathbb{C}$ such that this integral is well-defined.
\begin{lem}
 The Fourier transform in $\D$, $\widetilde{\textbf{W}}^{sh}(\lambda,b)$ of $\textbf{W}^{sh}$ does not depend upon the variable $b\in\partial\D$.
\end{lem}
\begin{proof}
 For  all $\lambda\in\R$ and $b=e^{i\theta}\in\partial\D$,
\bqs
\widetilde{\textbf{W}}^{sh}(\lambda,b)=\int_{\D}\textbf{W}^{sh}(z)e^{(-i\lambda +1)<z,b>}dm(z).
\eqs
We recall that for all $\phi\in\R$ $r_{\phi}$ is the rotation of angle $\phi$ and we have $\textbf{W}^{sh}(r_{\phi}\cdot z)=\textbf{W}^{sh}(z)$, $dm(z)=dm(r_{\phi}\cdot z)$ and $<z,b>=<r_{\phi}\cdot z,r_{\phi}\cdot b>$, then:
\bqs
\widetilde{\textbf{W}}^{sh}(\lambda,b)=\int_{\D}\textbf{W}^{sh}(r_{-\theta}\cdot z)e^{(-i\lambda +1)<r_{-\theta}\cdot z,1>}dm(z)=\int_{\D}\textbf{W}^{sh}(z)e^{(-i\lambda +1)<z,1>}dm(z)\overset{def}{=}\widetilde{\textbf{W}}^{sh}(\lambda)
\eqs
\end{proof}
We now introduce two functions that enjoy some nice properties with respect to the Hyperbolic Fourier transform and are eigenfunctions of the linear operator $\mG_{l}^{sh}$.
\begin{prop}\label{prop:eigen}
 Let $e_{\lambda,b}(z)=e^{(-i\lambda +1)<z,b>}$ and $\Phi_{\lambda}(z)=\int_{\partial\D}e^{(i\lambda +1)<z,b>}db$ then:
\begin{itemize}
 \item $\mG_{l}^{sh}(e_{\lambda,b})=\widetilde{\textbf{W}}^{sh}(\lambda)e_{\lambda,b}$
 \item $\mG_{l}^{sh}(\Phi_{\lambda})=\widetilde{\textbf{W}}^{sh}(\lambda)\Phi_{\lambda}$
\end{itemize}
\end{prop}
\begin{proof}
We begin with $b=1\in\partial\D$ and use the horocyclic coordinates. We use the same changes of variables as in lemma \ref{lem:invariance}:
\bqs
\mG_{l}^{sh}(e_{\lambda,1})(n_sa_t.O)=
\int_{\R^2}W^{sh}(d_2(n_{s}a_{t}\cdot O,n_{s'}a_{t'}\cdot O))e^{(-i\lambda -1)t'}dt'ds'
\eqs
\bqs
=\int_{\R^2}W^{sh}(d_2(n_{s-s'}a_{t}\cdot O,a_{t'}\cdot O))e^{(-i\lambda -1)t'}dt'ds'
\eqs
\bqs
=\int_{\R^2}W^{sh}(d_2(a_{t}n_{-x}\cdot O,a_{t'}\cdot O))e^{(-i\lambda -1)t'+2t}dt'dx=\int_{\R^2}W^{sh}(d_2(O,n_{x}a_{t'-t}\cdot O))e^{(-i\lambda -1)t'+2t}dt'dx
\eqs
\bqs
=\int_{\R^2}W^{sh}(d_2(O,n_{x}a_{y}\cdot O))e^{(-i\lambda -1)(y+t)+2t}dydx=e^{(-i\lambda +1)<n_sa_t\cdot O,1>}\widetilde{\textbf{W}}^{sh}(\lambda)
\eqs
By rotation, we obtain the property for all $b\in\partial\D$.

For the second property \cite[Lemma 4.7]{helgason:00} shows that:
\bqs
\textbf{W}^{sh}\ast \Phi_{\lambda}(z)=\int_{\partial\D}e^{(i\lambda +1)<z,b>}\widetilde{\textbf{W}}^{sh}(\lambda)db=\Phi_{\lambda}(z)\widetilde{\textbf{W}}^{sh}(\lambda)
\eqs

\end{proof}

A consequence of this proposition is the following lemma.

\begin{lem}\label{lem:noncompact}
 The linear operator $\mG_l^{sh}$ is not compact and for all $\mu\geq 0$, the nonlinear operator $\mG_{\mu}^{sh}$ is not compact.
\end{lem}

\begin{proof}
 The previous proposition \ref{prop:eigen} shows that $\mG_{l}^{sh}$ has a continuous spectrum which iimplies that is not a compact operator.

Let $U$ be in $\mF^{sh}$, for all $V\in\mF^{sh}$ we differentiate $\mG_{\mu}^{sh}$ and compute its Frechet derivative:
\bqs
D\left(\mG_{\mu}^{sh}\right)_{U}(V)(z)=\int_{\D}W^{sh}(d_2(z,z'))S'(U(z'))V(z')dm(z')
\eqs
If we assume further that $U$ does not depend upon the space variable $z$, $U(z)=U_0$ we obtain:
\bqs
D\left(\mG_{\mu}^{sh}\right)_{U_0}(V)(z)=S'(U_0)\mG_{l}^{sh}(V)(z)
\eqs
If $\mG_{\mu}^{sh}$ was a compact operator then its Frechet derivative $D\left(\mG_{\mu}^{sh}\right)_{U_0}$ would also be a compact operator, but it is impossible. As a consequence, $\mG_{\mu}^{sh}$ is not a compact operator.
 \end{proof}

\subsection{The convolution form of the operator $\mG_{\mu}$ in the general case}

We adapt the ideas presented in the previous section in order to deal with the general case. We recall that if $H$ is the group of positive real numbers with multiplication as operation, then the Haar measure $dh$ is given by $\dfrac{dx}{x}$. For two functions $f_1,f_2$ in $L^1(\D\times\R_{*}^{+})$ we define the convolution $\star$ by:
\bqs
(f_1\star f_2)(z,\Delta)\overset{def}{=}\int_G\int_H f_1(g\cdot O,h\cdot 1)f_2(g^{-1}\cdot z,h^{-1}\cdot \Delta)dgdh
\eqs
We recall that we have set by definition: $\mbW(z,\Delta)=W(d_2(z,0),|\log(\Delta)|)$.
\begin{prop}
 For all $\mu\geq0$ and $V\in\mF$ we have:
\bq
\mG_{\mu}(V)=\textbf{W}\star S(\mu V),\quad \mG_{l}(V)=\textbf{W}\star V\text{ and }\mG_{\infty}(V)=\textbf{W}\star H(V)
\eq
\end{prop}
\begin{proof}
 Let $(z,\Delta)$ be in $\D\times\R_{*}^{+}$. We follow the same ideas as in proposition \ref{prop:convolution} and prove only the first result. We have
\bqs
\begin{array}{lcl}
\mG_{\mu}(V)(z,\Delta)&=&\int_{\D\times\R_{*}^{+}}W\left(d_2(z,z'),\left|\log(\frac{\Delta}{\Delta'})\right|\right)S(\mu V(z',\Delta'))\,\dfrac{d\Delta'}{\Delta'}dm(z')\\
&=&\int_{G}\int_{\R_{*}^{+}}W\left(d_2(g^{-1}\cdot z,O),\left|\log\left(\frac{\Delta}{\Delta'}\right)\right|\right)S(\mu V(g\cdot O,\Delta'))\,dg\,\dfrac{d\Delta'}{\Delta'}\\
&=&\int_{G}\int_{H}W\left(d_2(g^{-1}\cdot z,O),\left|\log\left(h^{-1}\cdot\Delta\right)\right|\right)S(\mu V(g\cdot O,h\cdot 1))\,dg\,dh\\
&=&\textbf{W}\star S(\mu V)(z,\Delta)
\end{array}
\eqs
\end{proof}
We next assume further that the function $\mbW$ is separable in $z$, $\Delta$ and more precisely that $\mbW(z,\Delta)=\mbW_1(z)\mbW_2(\log(\Delta))$ where $\mbW_1(z)=W_1(d_2(z,0))$ and $\mbW_2(\log(\Delta))=W_2(\left|\log(\Delta)\right|)$ for all $(z,\Delta)\in \D\times\R_{*}^{+}$. The following proposition is an echo to proposition \ref{prop:eigen}.
\begin{prop}
Let $e_{\lambda,b}(z)=e^{(-i\lambda +1)<z,b>}$, $\Phi_{\lambda}(z)=\int_{\partial\D}e^{(i\lambda +1)<z,b>}db$ and $h_{\xi}(\Delta)=e^{i\xi\log(\Delta)}$ then:
\begin{itemize}
 \item $\mG_{l}(e_{\lambda,b}h_{\xi})=\widetilde{\mbW}_1(\lambda)\widehat{\mbW}_2(\xi) e_{\lambda,b}h_{\xi}$
 \item $\mG_{l}(\Phi_{\lambda}h_{\xi})=\widetilde{\mbW}_1(\lambda)\widehat{\mbW}_2(\xi)\Phi_{\lambda}h_{\xi}$
\end{itemize}
where $\widehat{\mbW}_2$ is the usual Fourier transform of $\mbW_2$.
\end{prop}
\begin{proof}
The proof of this proposition is exactly the same as for proposition \ref{prop:eigen}. Indeed:
\bqs
\begin{array}{lcl}
\mG_{l}(e_{\lambda,b}h_{\xi})(z,\Delta)&=&\mbW_1\ast e_{\lambda,b}(z)\int_{\R_{*}^{+}}\mbW_2\left(\log\left(\frac{\Delta}{\Delta'}\right)\right)e^{i\xi\log(\Delta')}\dfrac{d\Delta'}{\Delta'}\\
&=&\mbW_1\ast e_{\lambda,b}(z)\left(\int_{\R}\mbW_2(y)e^{-i\xi y}dy\right) e^{i\xi\log(\Delta)}
\end{array}
\eqs
\end{proof}

A straightforward consequence of this proposition is an extension of lemma \ref{lem:noncompact} to the general case:
\begin{lem}
 The linear operator $\mG_l^{sh}$ is not compact and for all $\mu\geq 0$, the nonlinear operator $\mG_{\mu}^{sh}$ is not compact.
\end{lem}

\subsection{The set of the solutions of (\ref{eq:genstate})}

Let $\mB_{\mu}$ the set of the solutions of \eqref{eq:genstate} for a given slope parameter $\mu$:
\bqs
\mB_{\mu}=\{ V\in\mF | -\alpha V+\mG_{\mu}(V)+I=0 \}
\eqs
We have the following proposition.
\begin{prop} 
If the input current $I$ is equal to a constant $I_0$, i.e.  does not depend upon the variables $(z,\Delta)$ then for all $\mu\in\R^+$, $\mB_{\mu}\neq\emptyset$.
In the general case $I\in \mF$, if the condition $\mu S'_m W^{g}_0<\alpha$ is satisfied, then $Card(\mB_{\mu})=1$.
\end{prop}
\begin{proof}
Due to the properties of the sigmoid function, there always exists a constant solution in the case where $I$ is constant. In the general case where $I \in \mF$, the statement is a direct application of the Banach fixed point theorem, as in \cite{faugeras-veltz-etal:09}.
\end{proof}
\newtheorem{rmk}{Remark}[subsection]
\begin{rmk}
 It should be clear that if the input current does not depend upon the variables $(z,\Delta)$ and if the condition $\mu S'_m W^{g}_0<\alpha$ is satisfied, then there exists a unique stationary solution which, in effect, does not depend upon the variables $(z,\Delta)$. \\
If on the other hand the input current does depend upon these variables, is invariant under the action of a subgroup of $U(1,1)$, the group of the isometries of $\D$ (see \eqref{section:isom}), and the condition $\mu S'_m W^{g}_0<\alpha$ is satisfied, then the unique stationary solution will also be invariant under the action of the same subgroup.\\
When the condition $\mu S'_m W^{g}_0<\alpha$ is satisfied we call primary stationary solution the unique solution in $\mB_{\mu}$.
\end{rmk}

\subsection{Stability of the primary stationary solution}\label{subsection:stability}

In this subsection we show that the condition $\mu S'_m W^{g}_0<\alpha$ guarantees the stability of the primary stationary solution to \eqref{eq:neurco}.
\begin{thm} We suppose that $I\in\mF$ and that the condition $\mu S'_m W^{g}_0<\alpha$ is satisfied, then the associated primary stationary solution of \eqref{eq:neurco} is asymtotically stable.
\end{thm}
\begin{proof}
 Let $V_{\mu}^0$ be the primary stationary solution of \eqref{eq:neurco}, as $\mu S'_m W^{g}_0<\alpha$ is satisfied. Let also $V_{\mu}$ be the unique solution of the same equation with some initial condition $V_{\mu}(0)=\phi\in\mF$, see theorem \ref{thm:exun}. We introduce a new function $X=V_{\mu}-V_{\mu}^0$ which satisfies:
\bqs
\left\{ \begin{array}{ll}
\partial_t 
X(z,\Delta,t)=-\alpha X(z,\Delta,t)+\int_{\D\times\R_{*}^{+}}W_m\left(d_2(z,z'),\left|\log(\frac{\Delta}{\Delta'})\right|\right)\Theta(X(z',\Delta',t))\frac{d\Delta'}{\Delta'}dm(z')\\
X(z,\Delta,0)=\phi(z,\Delta)-V_{\mu}^0(z,\Delta)
\end{array}
\right.
\eqs
where $W_m\left(d_2(z,z'),\left|\log(\frac{\Delta}{\Delta'})\right|\right)=S'_m W\left(d_2(z,z'),\left|\log(\frac{\Delta}{\Delta'})\right|\right)$ and the vector $\Theta(X(z,\Delta,t))$ is given by $\Theta(X(z,\Delta,t))=\underline{S}(\mu V_{\mu}(z,\Delta,t))-\underline{S}(\mu V_{\mu}^0(z,\Delta))$ with $\underline{S}=(S'_m)^{-1}S$. 
We note that, because of the definition of $\Theta$ and the mean value theorem $|\Theta(X(z,\Delta,t))| \leq \mu |X(z,\Delta,t)|$. This implies that $|\Theta(r)|\leq |r|$ for all $r\in\R$.
\bqs
\partial_t 
X(z,\Delta,t)=-\alpha X(z,\Delta,t)+\int_{\D\times\R_{*}^{+}}W_m\left(d_2(z,z'),\left|\log(\frac{\Delta}{\Delta'})\right|\right)\Theta(X(z',\Delta',t))\frac{d\Delta'}{\Delta'}dm(z')
\eqs

\bqs
\Rightarrow \partial_t \bigg(e^{\alpha t}X(z,\Delta,t)\bigg)=e^{\alpha t}\int_{\D\times\R_{*}^{+}}W_m\left(d_2(z,z'),\left|\log(\frac{\Delta}{\Delta'})\right|\right)\Theta(X(z',\Delta',t))\frac{d\Delta'}{\Delta'}dm(z')
\eqs
\bqs
\Rightarrow X(z,\Delta,t)= e^{-\alpha t}X(z,\Delta,0)+\int_0^te^{-\alpha (t-u)}\int_{\D\times\R_{*}^{+}}W_m\left(d_2(z,z'),\left|\log(\frac{\Delta}{\Delta'})\right|\right)\Theta(X(z',\Delta',u))\frac{d\Delta'}{\Delta'}dm(z')du
\eqs
\bqs
\Rightarrow |X(z,\Delta,t)|\leq e^{-\alpha t}|X(z,\Delta,0)|+\mu \int_0^te^{-\alpha (t-u)}\int_{\D\times\R_{*}^{+}}\left|W_m\left(d_2(z,z'),\left|\log(\frac{\Delta}{\Delta'})\right|\right)\right|\left|X(z',\Delta',u)\right|\frac{d\Delta'}{\Delta'}dm(z')du
\eqs
\bqs
\Rightarrow \|X(t)\|_{\infty}\leq e^{-\alpha t}\|X(0)\|_{\infty}+\mu W^{g}_0 S'_m\int_0^te^{-\alpha (t-u)}\|X(u)\|_{\infty}du
\eqs
If we set: $G(t)=e^{\alpha t}\|X(t)\|_{\infty}$, then we have:
\bqs
G(t)\leq G(0)+\mu W^{g}_0 S'_m\int_0^tG(u)du
\eqs
and $G$ is continuous for all $t\geq 0$.
The Gronwall inequality implies that:
\bqs
G(t)\leq G(0)e^{\mu W^{g}_0 S'_m t}
\eqs
\bqs
\Rightarrow \|X(t)\|_{\infty}\leq e^{(\mu W^{g}_0 S'_m-\alpha) t}\|X(0)\|_{\infty},
\eqs
and the conclusion follows.

\end{proof}

\section{Spatially localised bumps in the high gain limit}
\label{section:bump}

In many models of working memory, transient stimuli are encoded by feature-selective persistent neural activity. Such stimuli are imagined to induce the formation of a spatially localised bump of persistent activity which coexists with a stable uniform state. As an example, Camperi and Wang \cite{camperi-wang:98} have proposed and studied a network model of visuo-spatial working memory in prefontal cortex adapted from the ring model of orientation of Ben-Yishai and colleagues \cite{ben-yishai-bar-or-etal:95}. It is therefore natural to study the emergence of spatially localised bumps for the structure tensor model in a hypercolumn of V1. We only deal with the reduced case of equation \eqref{eq:semihom} and keep the general case for future work.\\
In order to construct exact bump solutions, we consider the high gain limit $\mu\rightarrow\infty$ of the sigmoid function. As above we denote by $H$ the Heaviside function defined by $H(x)=1$ for $x\geq0$ and $H(x)=0$ otherwise. Equation \eqref{eq:semihom} is rewritten as:

\bq
\partial_t V(z,t)=-\alpha V(z,t)+\int_{\D} W(z,z')H(V(z',t)-\kappa)\text{dm}(z')+I(z,t)
\label{eq:semimod}
\eq
\bqs
=-\alpha V(z,t)+\int_{\{ z'\in\D | V(z',t)\geq\kappa \}} W(z,z')\text{dm}(z')+I(z)
\eqs
We make the assumption that the system is spatially homogeneous that is, the external input $I$ does not depend upon the variables $t$ and the connectivity function depends only on the hyperbolic distance between two points of $\D$: $W(z,z')=W(d_2(z,z'))$. We also introduce a threshold $\kappa$ to shift the zero of the Heaviside function.

\subsection{Stationary pulses}
Our aim is to construct a hyperbolic radially symmetric stationary pulse. Let us first consider a general stationary pulse:
\bqs
\alpha V(z)=\int_{\{ z'\in\D | V(z')\geq \kappa \}} W(z,z')\text{dm}(z')+I(z)
\eqs
We assume  that the set $K=\{ z\in\D | V(z)\geq\kappa \}\subset\D$ is compact. We note $M(z,K)$ the integral $\int_{K} W(z,z')\text{dm}(z')$. The relation $V(z)=\kappa$ holds for all $z\in\partial K$.

In order to calculate $M$, we use the Fourier transform. First we rewrite $M$ as a convolution product:
\bqs
M(z,K)=\int_{K} W(z,z')\text{dm}(z')=\int_{\D} W(z,z')\mathds{1}_{K}(z')\text{dm}(z')=\mbW \ast \mathds{1}_{K}(z)
\eqs
where $\mbW(z)\overset{def}{=}W(d_2(z,0))$.\\
In \cite{helgason:00}, Helgason proves an inversion formula for the hyperbolic Fourier transform and we apply this result to $\mbW$.
\bq
\mbW(z)=\frac{1}{4\pi}\int_{\R}\widetilde{\mbW}(\lambda)\Phi_{\lambda}(z)\lambda \tanh(\frac{\pi}{2}\lambda)d\lambda
\label{eq:fourtrans}
\eq
Then,
\bqs
M(z,K)=\mbW \ast \mathds{1}_{K}(z)=\frac{1}{4\pi}\int_{\R}\widetilde{\mbW}(\lambda)\Phi_{\lambda}\ast\mathds{1}_{K}(z)\lambda \tanh(\frac{\pi}{2}\lambda)d\lambda
\eqs
It appears that the study of $M(z,K)$ consists in calculating the convolution product $\Phi_{\lambda}\ast\mathds{1}_{K}(z)$.\\
Let $z=h\cdot O$ for $h\in G$ we have:
\bqs
\Phi_{\lambda}\ast\mathds{1}_{K}(z)=\int_{G}\mathds{1}_{K}(g\cdot O)\Phi_{\lambda}(g^{-1}\cdot z)dg
\eqs
\bqs
=\int_{G}\mathds{1}_{K}(g\cdot O)\Phi_{\lambda}(g^{-1}h\cdot O)dg
\eqs
for all $g,h\in G$, $\Phi_{\lambda}(g^{-1}h\cdot O)=\Phi_{\lambda}(h^{-1}g\cdot O)$ so that:
\bqs
\Phi_{\lambda}\ast\mathds{1}_{K}(z)=\int_{G}\mathds{1}_{K}(g\cdot O)\Phi_{\lambda}(h^{-1}g\cdot O)dg=\int_{\D}\mathds{1}_{K}(z')\Phi_{\lambda}(h^{-1}\cdot z')dm(z')
\eqs
\bqs
=\int_{K}\Phi_{\lambda}(h^{-1}\cdot z')dm(z')
\eqs
\subsubsection{Study of $M(z,K)$ when $K=B_h(0,\omega)$}
We now consider the special case where $K$ is a hyperbolic disk centered at the origin of hyperbolic radius $\omega$, noted $B_h(0,\omega)$.
\paragraph{First step}
We start by giving an explicit formula for a useful integral.
\begin{lem}\label{lem:hypergeometric}
 For all $\omega>0$ the following formula holds:
\bqs
\int_{B_h(0,\omega)}\Phi_{\lambda}(z)dm(z)=\pi\sinh(\omega)^2\cosh(\omega)^2\Phi_{\lambda}^{(1,1)}(\omega)
\eqs
where $B_h(0,\omega)$ is the hyperbolic ball and:
\bqs
\Phi_{\lambda}^{(\alpha,\beta)}(\omega)=F(\frac{1}{2}(\rho+i\lambda),\frac{1}{2}(\rho-i\lambda);\alpha+1;-\sinh(\omega)^2),
\eqs
where $\alpha+\beta+1=\rho$ and $F$ is the hypergeometric function of first kind.\\
\end{lem}

\begin{proof}
We write $z$ in hyperbolic polar coordinates, $z=\tanh(r)e^{i\theta}$ (see appendix \ref{section:isom}). We have:
\bqs
\int_{B_h(0,\omega)}\Phi_{\lambda}(z)dm(z)=\frac{1}{2}\int_0^\omega\int_0^{2\pi}\Phi_{\lambda}(\tanh(r)e^{i\theta})\sinh(2r)\,dr\,d\theta
\eqs
Because of the above definition of $\Phi_{\lambda}$, this reduces to
\bqs
\pi \int_0^\omega\Phi_{\lambda}(\tanh(r))\sinh(2r)\,dr
\eqs
In \cite{helgason:00} Helgason proved that:
\bqs
\Phi_{\lambda}(\tanh(r))=F(\nu,1-\nu;1;-\sinh(r)^2)
\eqs
with $\nu=\frac{1}{2}(1+i\lambda)$. We then use the formula obtained by Erdelyi in \cite{erdelyi:85}:
\bqs
F(\nu,1-\nu;1;z)=\frac{d}{dz}\bigg(zF(\nu,1-\nu;2;z) \bigg)
\eqs
Using some simple hyperbolic trigonometry formulae we obtain:
\bqs
\sinh(2r)F(\nu,1-\nu;1;-\sinh(r)^2)=\frac{d}{dr}\bigg(\sinh(r)^2F(\nu,1-\nu;2;-\sinh(r)^2) \bigg),
\eqs
from which we deduce
\bqs
\int_{B_h(0,\omega)}\Phi_{\lambda}(z)dm(z)=\pi \sinh(\omega)^2 F(\nu,1-\nu;2;-\sinh(\omega)^2)
\eqs
Finally we use the equality shown in \cite{erdelyi:85}:
\bqs
F(a,b;c;z)=(1-z)^{c-a-b}F(c-a,c-b;c;z)
\eqs
In our case we have: $a=\nu,b=1-\nu,c=2$ and $z=-\sinh(\omega)^2$, so $2-\nu=\frac{1}{2}(3-i\lambda)$, $1+\nu=\frac{1}{2}(3+i\lambda)$. We obtain
\bqs
\int_{B_h(0,\omega)}\Phi_{\lambda}(z)dm(z)=\pi \sinh(\omega)^2\cosh(\omega)^2 F\bigg(\frac{1}{2}(3-i\lambda),\frac{1}{2}(3+i\lambda);2;-\sinh(\omega)^2\bigg).
\eqs
Since Hypergeometric functions are symmetric with respect to the first two variables:
\bqs
F(a,b;c;z)=F(b,a;c;z),
\eqs
we write
\bqs
F\bigg(\frac{1}{2}(3-i\lambda),\frac{1}{2}(3+i\lambda);2;-\sinh(\omega)^2\bigg)=F\bigg(\frac{1}{2}(3+i\lambda),\frac{1}{2}(3-i\lambda);2;-\sinh(\omega)^2\bigg)=\Phi_{\lambda}^{(1,1)}(\omega),
\eqs
which yields the announced formula
\bqs
\int_{B_h(0,\omega)}\Phi_{\lambda}(z)dm(z)=\pi \sinh(\omega)^2\cosh(\omega)^2\Phi_{\lambda}^{(1,1)}(\omega)
\eqs
\end{proof}

\paragraph{Second step} In this paragraph, we show that the mapping $z=h\cdot O=\tanh(r)e^{i\theta}\rightarrow\int_{K}\Phi_{\lambda}(h^{-1}\cdot z')dm(z')$ is a radial function, i.e. it depends only upon the variable $r$. We then show that $\int_{B_h(0,\omega)}\Phi_{\lambda}(a_{-r}\cdot z')dm(z')=\Phi_{\lambda}(a_{r}\cdot O)\int_{B_h(0,\omega)}e^{(i\lambda +1)<z',1>}dm(z')$.\\

\begin{prop}
 If $z=h \cdot O$ and $z$ is written $\tanh(r)e^{i\theta}$ with $r=d_2(z,O)$ in hyperbolic polar coordinates the integral $\int_{B_h(0,\omega)}\Phi_{\lambda}(h^{-1}\cdot z')dm(z')$ depends only upon the variable $r$.
\end{prop}

\begin{proof}
 If $z=\tanh(r)e^{i\theta}$, then $z=\text{rot}_{\theta}a_{r}\cdot O$ and $h^{-1}=a_{-r}\text{rot}_{-\theta}$. Similarly $z'=\text{rot}_{\theta'}a_{r'}\cdot O$. We can write
\bqs
\int_{B_h(0,\omega)}\Phi_{\lambda}(h^{-1}\cdot z')dm(z')=\frac{1}{2}\int_0^\omega\int_0^{2\pi}\Phi_{\lambda}(a_{-r}\text{rot}_{\theta'-\theta}a_{r'}\cdot O)\sinh(2r')dr'd\theta'
\eqs

\bqs
=\frac{1}{2}\int_0^\omega\int_0^{2\pi}\Phi_{\lambda}(a_{-r}\text{rot}_{\psi}a_{r'}\cdot O)\sinh(2r')dr'd\psi=\int_{B_h(0,\omega)}\Phi_{\lambda}(a_{-r}\cdot z')dm(z'),
\eqs
which, as announced, is only a function of $r$.
\end{proof}

We now give an explicit formula for the integral $\int_{B_h(0,\omega)}\Phi_{\lambda}(a_{-r}\cdot z')dm(z')$. We first recall a formula from \cite{helgason:00}.

\begin{lem}
For all $g\in G$ the following equation holds:
 \bqs
\Phi_{\lambda}(g^{-1}\cdot z)=\int_{\partial\D}e^{(-i\lambda +1)<g\cdot O,b>}e^{(i\lambda +1)<z,b>}db
\eqs
\end{lem}
\begin{proof}
 See \cite{helgason:00}.
\end{proof}
It follows immediately that for all $z\in\D$ and $r\in\R$ we have:
\bqs
\Phi_{\lambda}(a_{-r}\cdot z)=\int_{\partial\D}e^{(-i\lambda +1)<a_{r}\cdot O,b>}e^{(i\lambda +1)<z,b>}db
\eqs
We integrate this formula over the hyperbolic ball $B_h(0,\omega)$ which gives:
\bqs
\int_{B_h(0,\omega)}\Phi_{\lambda}(a_{-r}\cdot z')dm(z')=\int_{B_h(0,\omega)}\Bigg(\int_{\partial\D}e^{(-i\lambda +1)<a_{r}\cdot O,b>}e^{(i\lambda +1)<z',b>}db\Bigg)dm(z'),
\eqs
and we exchange the order of integration:
\bqs
=\int_{\partial\D}e^{(-i\lambda +1)<a_{r}\cdot O,b>}\Bigg(\int_{B_h(0,\omega)}e^{(i\lambda +1)<z',b>}dm(z')\Bigg)db.
\eqs
We note that the integral $\int_{B_h(0,\omega)}e^{(i\lambda +1)<z',b>}dm(z')$ does not depend upon the variable $b=e^{i\phi}$. Indeed:
\bqs
\int_{B_h(0,\omega)}e^{(i\lambda +1)<z',b>}dm(z')=\frac{1}{2}\int_{0}^{\omega}\int_{0}^{2\pi}\bigg(\frac{1-\tanh(x)^2}{|\tanh(x)e^{i\theta}-e^{i\phi}|^2}\bigg)^{\frac{i\lambda +1}{2}}\sinh(2x)dxd\theta
\eqs
\bqs
=\frac{1}{2}\int_{0}^{\omega}\int_{0}^{2\pi}\bigg(\frac{1-\tanh(x)^2}{|\tanh(x)e^{i(\theta-\phi)}-1|^2}\bigg)^{\frac{i\lambda +1}{2}}\sinh(2x)dxd\theta
\eqs
\bqs
=\frac{1}{2}\int_{0}^{\omega}\int_{0}^{2\pi}\bigg(\frac{1-\tanh(x)^2}{|\tanh(x)e^{i\theta'}-1|^2}\bigg)^{\frac{i\lambda +1}{2}}\sinh(2x)dxd\theta',
\eqs
and indeed the integral does not depend upon the variable $b$:
\bqs
\int_{B_h(0,\omega)}e^{(i\lambda +1)<z',b>}dm(z')=\int_{B_h(0,\omega)}e^{(i\lambda +1)<z',1>}dm(z').
\eqs
Finally, we can write:
\bqs
\int_{B_h(0,\omega)}\Phi_{\lambda}(a_{-r}\cdot z')dm(z')=\int_{\partial\D}e^{(-i\lambda +1)<a_{r}\cdot O,b>}db\int_{B_h(0,\omega)}e^{(i\lambda +1)<z',1>}dm(z')
\eqs
\bqs
=\Phi_{-\lambda}(a_{r}\cdot O)\int_{B_h(0,\omega)}e^{(i\lambda +1)<z',1>}dm(z')=\Phi_{\lambda}(a_{r}\cdot O)\int_{B_h(0,\omega)}e^{(i\lambda +1)<z',1>}dm(z'),
\eqs
because $\Phi_{\lambda}=\Phi_{-\lambda}$ (as solutions of the same equation).\\
This completes the proof that:
\bq
\int_{B_h(0,\omega)}\Phi_{\lambda}(a_{-r}\cdot z')dm(z')=\Phi_{\lambda}(a_{r}\cdot O)\int_{B_h(0,\omega)}e^{(i\lambda +1)<z',1>}dm(z')
\label{eq:interm}
\eq

\paragraph{The main result} 
At this point we have proved the following proposition.
\begin{prop}
If $K=B_h(0,\omega)$ and $z=\text{rot}_{\theta}a_{r}\cdot O\in K$, $M(z,K)$ is given by the following formula:
\bq
M(z,B_h(0,\omega))\overset{def}{=}\mM(r,\omega)=\frac{1}{4\pi}\int_{\R}\widetilde{\mbW}(\lambda)\Phi_{\lambda}(a_{r}\cdot O)\Psi_{\lambda}(\omega)\lambda \tanh(\frac{\pi}{2}\lambda)d\lambda
\label{eq:mdisk}
\eq
where
\bq
\Psi_{\lambda}(\omega)\overset{def}{=}\int_{B_h(0,\omega)}e^{(i\lambda +1)<z',1>}dm(z')
\label{eq:mldisk}
\eq
\end{prop}

\noindent
We are now in a position to obtain an analytic form for $\mM(r,\omega)$. It is given in the following theorem.
\begin{thm}
For all $(r,\omega)\in\R^{+}\times\R^{+}$:
\bq
 \mM(r,\omega)=\frac{1}{4}\sinh(\omega)^2\cosh(\omega)^2\int_{\R}\widetilde{\mbW}(\lambda)\Phi_{\lambda}^{(0,0)}(r)\Phi_{\lambda}^{(1,1)}(\omega)\lambda \tanh(\frac{\pi}{2}\lambda)d\lambda
\label{eq:formula}
\eq
\end{thm}
\begin{proof}
We prove 
 \bqs
\Psi_{\lambda}(\omega)=\int_{B_h(0,\omega)}\Phi_{\lambda}(z)dm(z).
\eqs
Indeed, in hyperbolic polar coordinates, we have:
\bqs
\begin{array}{lcl}
 \Psi_{\lambda}(\omega)&=&\frac{1}{2}\int_{0}^{\omega}\int_{0}^{2\pi}e^{(i\lambda +1)<rot_{\theta}a_{r}\cdot O,1>}\sinh(2r)\,dr\,d\theta\\
&&\\
&=&\frac{1}{2}\int_{0}^{\omega}\int_{0}^{2\pi}e^{(i\lambda +1)<a_{r}\cdot O,e^{-i\theta}>}\sinh(2r)\,dr\,d\theta\\
&&\\
&=&\pi\int_{0}^{\omega}\int_{\partial\D}e^{(i\lambda +1)<a_{r}\cdot O,b>}db\sinh(2r)\,dr\\
&&\\
&=&\pi\int_{0}^{\omega}\Phi_{\lambda}(a_{r}\cdot O)\sinh(2r)\,dr
\end{array}
\eqs
On the other hand:
\bqs
\begin{array}{lcl}
\int_{B_h(0,\omega)}\Phi_{\lambda}(z)dm(z)&=&\frac{1}{2}\int_{0}^{\omega}\int_{0}^{2\pi}\Phi_{\lambda}(a_{r}\cdot O)\sinh(2r)\,dr\,d\theta\\
&&\\
&=&\pi\int_{0}^{\omega}\int_{0}^{2\pi}\Phi_{\lambda}(a_{r}\cdot O)\sinh(2r)\,dr
\end{array}
\eqs
This yields
\bqs
\Psi_{\lambda}(\omega)=\int_{B_h(0,\omega)}\Phi_{\lambda}(z)dm(z),
\eqs
and we  use lemma \eqref{lem:hypergeometric} to establish \eqref{eq:formula}.
\end{proof}

\paragraph{Discussion}

Let us point out that our result can be linked to the work of Folias and Bressloff in \cite{folias-bressloff:04}. They constructed a two-dimensional pulse for a general, radially symmetric synaptic weight function. They obtain a similar formal representation of the integral of the connectivity function $w$ over the disk $B(O,a)$ centered at the origin $O$ and of radius $a$. Using their notations,
\bqs
M(a,r)=\int_{0}^{2\pi}\int_{0}^{a}w(|\mathbf{r}-\mathbf{r}'|)r'\,dr'\,d\theta=2\pi a\int_{0}^{\infty}\breve{w}(\rho)J_{0}(r\rho)J_{1}(a\rho)\,d\rho
\eqs
where $J_{\nu}(x)$ is the Bessel function of the first kind and $\breve{w}$ is the real Fourier transform of $w$. In our case, instead of the Bessel function, we find $\Phi_{\lambda}^{(\nu,\nu)}(r)$ which is linked to the hypergeometric function of the first kind, as explained in lemma \ref{lem:hypergeometric}.

We next adapt the results proved by Folias and Bressloff in \cite{folias-bressloff:04} to the hyperbolic case.

\subsubsection{A hyperbolic radially symmetric stationary-pulse}

We note $V(r)$ a hyperbolic radially symmetric stationary-pulse solution of \eqref{eq:semimod} where $V$ depends only upon the variable $r$ and is such that:
\bqs
V(r)>\kappa, \quad r\in [0,\omega [,
\eqs
\bqs
V(\omega)=\kappa,
\eqs
\bqs
V(r)<\kappa, \quad r\in ]\omega,\infty[,
\eqs
and
\bqs
V(\infty)=0.
\eqs
Substituting into \eqref{eq:semimod} yields:
\bq
\alpha V(r)=\mM(r,\omega)+I(r)
\label{eq:radeq}
\eq
where $\mM(r,\omega)$ is defined in equation \eqref{eq:formula} and $I(r)=\mathcal{I}e^{-\frac{r^2}{2\sigma^2}}$ is a Gaussian input.\\
The condition for the existence of a stationary pulse is given by:
\bq
\alpha \kappa=\mM(\omega)+I(\omega)\overset{def}{=}N(\omega)
\label{eq:excond}
\eq
where
\bqs
\mM(\omega)\overset{def}{=}\mM(\omega,\omega)=\frac{1}{4}\sinh(\omega)^2\cosh(\omega)^2\int_{\R}\widetilde{\mbW}(\lambda)\Phi_{\lambda}^{(0,0)}(\omega)\Phi_{\lambda}^{(1,1)}(\omega)\lambda \tanh(\frac{\pi}{2}\lambda)d\lambda
\eqs
The function $N(\omega)$ is plotted in figure \ref{fig:Nomega} for a range of the input amplitude $\mathcal{I}$. The horizontal dashed lines indicate different values of $\alpha\kappa$, the points of intersection determine the existence of stationary pulse solutions.
\begin{figure}
 \centering
 \includegraphics[width=0.5\textwidth,bb=0 0 402 326]{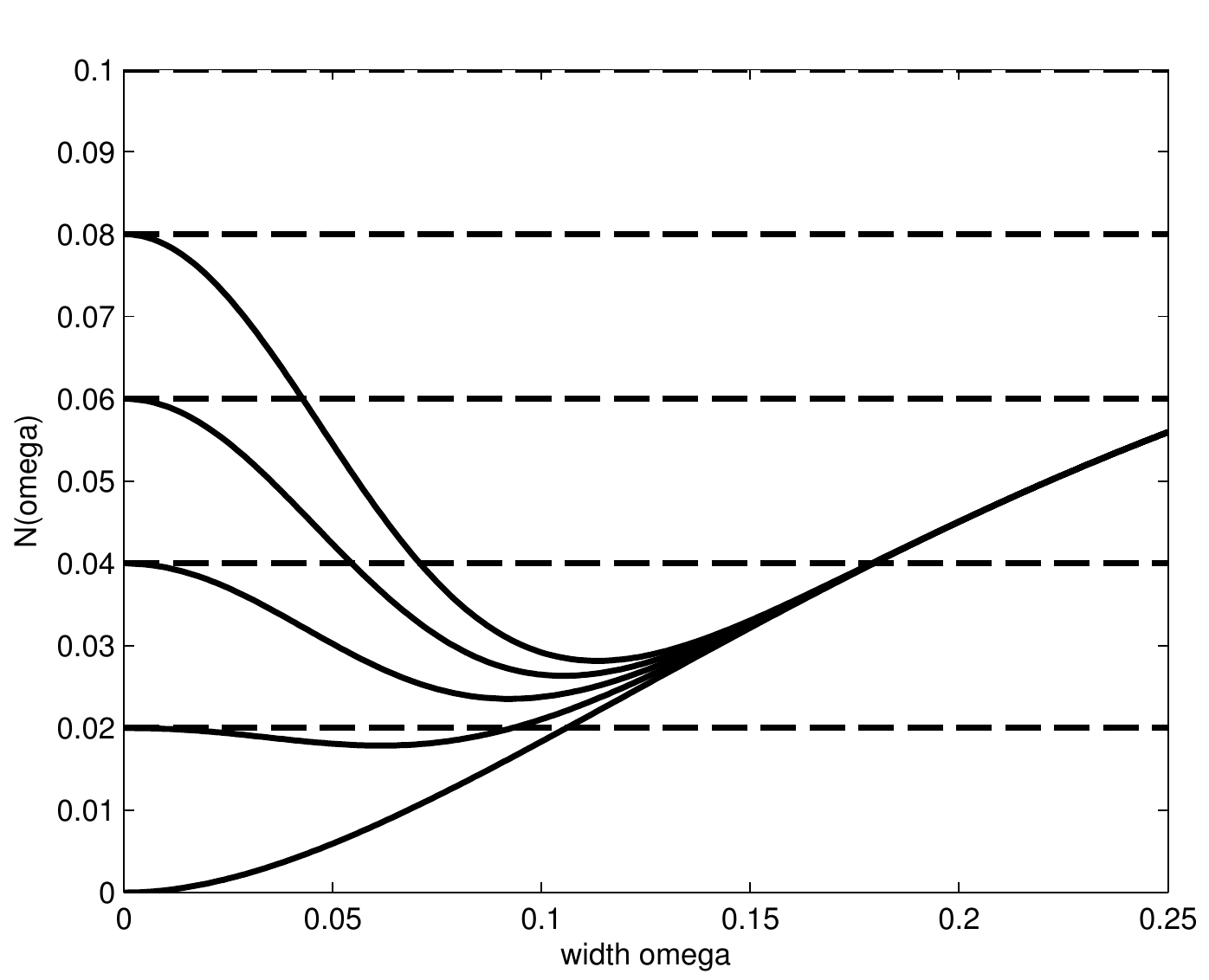}
 \caption{Plot of $N(\omega)$ defined in (\ref{eq:excond}) as a function of the pulse width $\omega$ for several values of the input amplitude $\mathcal{I}$ and for a fixed input width $\sigma=0.05$. The connectivity function has been set to $\tilde{w}(r)=e^{-\frac{r}{b}}$, with $b=0.2$.}
 \label{fig:Nomega}
\end{figure}

We now show that for a general monotonically decreasing weight function $\tilde{w}$, the function $\mM(r,\omega)$ is necessarily a monotonically decreasing function of $r$. This will ensure that the hyperbolic radially symmetric stationary-pulse solution \eqref{eq:radeq} is also a monotonically decreasing function of $r$ in the case of a Gaussian input. Differentiating $\mM$ with respect to $r$ yields:
\bqs
\frac{\partial \mM}{\partial r}(r,\omega)=\frac{1}{2}\int_{0}^{\omega}\int_{0}^{2\pi}\frac{\partial}{\partial r}\bigg(\tilde{w}(d_2(\tanh(r),\tanh(r')e^{i\theta}))\bigg)\sinh(2r')dr'd\theta
\eqs
We have to compute
\bqs
\frac{\partial}{\partial r}\bigg(\tilde{w}(d_2(\tanh(r),\tanh(r')e^{i\theta}))\bigg)=\tilde{w}'(d_2(\tanh(r),\tanh(r')e^{i\theta}))\frac{\partial}{\partial r}\bigg(d_2(\tanh(r),\tanh(r')e^{i\theta})\bigg).
\eqs
It is result of elementary hyperbolic trigonometry that
\bqs
d_2(\tanh(r),\tanh(r')e^{i\theta})=\tanh^{-1}\Bigg(\sqrt{\frac{\tanh(r)^2+\tanh(r')^2-2\tanh(r)\tanh(r')\cos(\theta)}{1+\tanh(r)^2\tanh(r')^2-2\tanh(r)\tanh(r')\cos(\theta)}}\Bigg)
\eqs
we let $\rho=\tanh(r)$, $\rho'=\tanh(r')$ and define
\bqs
F_{\rho',\theta}(\rho)=\frac{\rho^2+\rho'^2-2\rho\rho'\cos(\theta)}{1+\rho^2\rho'^2-2\rho\rho'\cos(\theta)}
\eqs
It follows that
\bqs
\frac{\partial}{\partial \rho}\tanh^{-1}\bigg(\sqrt{F_{\rho',\theta}(\rho)}\bigg)=\frac{\frac{\partial}{\partial \rho}F_{\rho',\theta}(\rho)}{2(1-F_{\rho',\theta}(\rho))\sqrt{F_{\rho',\theta}(\rho)}},
\eqs
and
\bqs
\frac{\partial}{\partial \rho}F_{\rho',\theta}(\rho)=\frac{2(\rho-\rho'\cos(\theta))+2\rho\rho'(\rho'-\rho\cos(\theta))}{(1+\rho^2\rho'^2-2\rho\rho'\cos(\theta))^2}
\eqs
We conclude that if $\rho>\tanh(\omega)$ then for all $0\leq \rho'\leq\tanh(\omega)$ and $0\leq \theta\leq 2\pi$
\bqs
2(\rho-\rho'\cos(\theta))+2\rho\rho'(\rho'-\rho\cos(\theta))>0,
\eqs
which implies $\mM(r,\omega)<0$ for $r>\omega$, since $\tilde{w}'<0$. 

To see that it is also negative for $r<\omega$, we differentiate equation \eqref{eq:formula} with respect to $r$:
\bqs
\frac{\partial \mM}{\partial r}(r,\omega)=\frac{1}{4}\sinh(\omega)^2\cosh(\omega)^2\int_{\R}\widetilde{\mbW}(\lambda)\frac{\partial}{\partial r}\Phi_{\lambda}^{(0,0)}(r)\Phi_{\lambda}^{(1,1)}(\omega)\lambda \tanh(\frac{\pi}{2}\lambda)d\lambda
\eqs
The following formula holds for the hypergeometric function (see Erdelyi in \cite{erdelyi:85}):
\bqs
 \frac{d}{dz}F(a,b;c;z)=\frac{ab}{c}F(a+1,b+1;c+1;z)
\eqs
It implies
\bqs
\frac{\partial}{\partial r}\Phi_{\lambda}^{(0,0)}(r)=-\frac{1}{2}\sinh(r)\cosh(r)(1+\lambda^2)\Phi_{\lambda}^{(1,1)}(r).
\eqs
Substituting in the previous equation giving $\frac{\partial \mM}{\partial r}$ we find:
\bqs
\frac{\partial \mM}{\partial r}(r,\omega)=-\frac{1}{64}\sinh(2\omega)^2\sinh(2r)\int_{\R}\widetilde{\mbW}(\lambda)(1+\lambda^2)\Phi_{\lambda}^{(1,1)}(r)\Phi_{\lambda}^{(1,1)}(\omega)\lambda \tanh(\frac{\pi}{2}\lambda)d\lambda,
\eqs
implying that:
\bqs
\textbf{sgn}(\frac{\partial \mM}{\partial r}(r,\omega))=\textbf{sgn}(\frac{\partial \mM}{\partial r}(\omega,r)).
\eqs
Consequently, $\frac{\partial \mM}{\partial r}(r,\omega)<0$ for $r<\omega$. Hence $V$ is monotonically decreasing in $r$ for any monotonically decreasing synaptic weight function $\tilde{w}$.

\subsubsection{Linear stability analysis}

We now analyse the evolution of small time-dependent perturbations of the hyperbolic stationary-pulse solution through linear stability analysis.

\paragraph{Spectral analysis of the linearized operator} Equation \eqref{eq:semimod} is linearized about the stationary solution $V(r)$ by introducing the time-depndent perturbation:
\bqs
v(z,t)=V(r)+\phi(z,t)
\eqs
This leads to the linear equation:
\bqs
\partial_t \phi(z,t)=-\alpha \phi(z,t)+\int_{\D} W(z,z')H'(V(r')-\kappa)\phi(z',t)\text{dm}(z').
\eqs
We separate variables by setting $\phi(z,t)=\phi(z)e^{\beta t}$ to obtain the equation:
\bqs
(\beta +\alpha)\phi(z)=\int_{\D} W(z,z')H'(V(r')-\kappa)\phi(z')\text{dm}(z')
\eqs
Introducing the hyperbolic polar coordinates $z=\tanh(r)e^{i\theta}$ and using the result:
\bqs
H'(V(r)-\kappa)=\delta(V(r)-\kappa)=\frac{\delta(r-\omega)}{|V'(\omega)|}
\eqs
we obtain:
\bqs
(\beta +\alpha)\phi(z)=\frac{1}{2}\int_{0}^{\omega}\int_{0}^{2\pi}W(\tanh(r)e^{i\theta},\tanh(r')e^{i\theta'})\frac{\delta(r'-\omega)}{|V'(\omega)|}\phi(\tanh(r')e^{i\theta'})\sinh(2r')dr'd\theta'
\eqs
\bqs
=\frac{\sinh(2\omega)}{2|V'(\omega)|}\int_{0}^{2\pi}W(\tanh(r)e^{i\theta},\tanh(\omega)e^{i\theta'})\phi(\tanh(\omega)e^{i\theta'})d\theta'
\eqs
With a slight abuse of notation we are led to study the solutions of the integral equation:
\bq
(\beta +\alpha)\phi(r,\theta)=\frac{\sinh(2\omega)}{2|V'(\omega)|}\int_{0}^{2\pi}\mathcal{W}(r,\omega;\theta'-\theta)\phi(\omega,\theta')d\theta'
\label{eq:intequ}
\eq
where:
\bqs
\mathcal{W}(r,\omega;\varphi)\overset{def}{=}\tilde{w}\circ\tanh^{-1}\Bigg(\sqrt{\frac{\tanh(r)^2+\tanh(\omega)^2-2\tanh(r)\tanh(\omega)\cos(\varphi)}{1+\tanh(r)^2\tanh(\omega)^2-2\tanh(r)\tanh(\omega)\cos(\varphi)}}\Bigg)
\eqs

\paragraph{Essential spectrum} If the function $\phi$ satisfies the condition
\bqs
\int_{0}^{2\pi}\mathcal{W}(r,\omega;\theta')\phi(\omega,\theta-\theta')d\theta'=0 \quad \forall r,
\eqs
then equation \eqref{eq:intequ} reduces to:
\bqs
\beta +\alpha=0
\eqs
yielding the eigenvalue:
\bqs
\beta=-\alpha<0
\eqs
This part of the essential spectrum is negative and does not cause instability.
\paragraph{Discrete spectrum} If we are not in the previous case we have to study the solutions of the integral equation (\ref{eq:intequ}).\\
This equation shows that $\phi(r,\theta)$ is completely determined by its values $\phi(\omega,\theta)$ on the circle of equation $r=\omega$. Hence, we need only to consider $r=\omega$, yielding the integral equation:
\bqs
(\beta +\alpha)\phi(\omega,\theta)=\frac{\sinh(2\omega)}{2|V'(\omega)|}\int_{0}^{2\pi}\mathcal{W}(\omega,\omega;\theta')\phi(\omega,\theta-\theta')d\theta'
\eqs
The solutions of this equation are exponential functions $e^{\gamma\theta}$, where $\gamma$ satisfies:
\bqs
(\beta +\alpha)=\frac{\sinh(2\omega)}{2|V'(\omega)|}\int_{0}^{2\pi}\mathcal{W}(\omega,\omega;\theta')e^{-\gamma\theta'}d\theta'
\eqs
By the requirement that $\phi$ is $2\pi$-periodic in $\theta$, it follows that $\gamma=in$, where $n\in\mathbb{Z}$. Thus the integral operator with kernel $\mathcal{W}$ has a discrete spectrum given by:
\bqs
(\beta_n+\alpha)=\frac{\sinh(2\omega)}{2|V'(\omega)|}\int_{0}^{2\pi}\mathcal{W}(\omega,\omega;\theta')e^{-in\theta'}d\theta'
\eqs
\bqs
=\frac{\sinh(2\omega)}{2|V'(\omega)|}\int_{0}^{2\pi}\tilde{w}\circ\tanh^{-1}\Bigg(\sqrt{\frac{2\tanh(\omega)^2(1-\cos(\theta'))}{1+\tanh(\omega)^4-2\tanh(\omega)^2\cos(\theta')}}\Bigg)e^{-in\theta'}d\theta'
\eqs
\bqs
=\frac{\sinh(2\omega)}{|V'(\omega)|}\int_{0}^{\pi}\tilde{w}\circ\tanh^{-1}\Bigg(\frac{2\tanh(\omega)\sin(\theta')}{\sqrt{(1-\tanh(\omega)^2)^2+4\tanh(\omega)^2\sin(\theta')^2}}\Bigg)e^{-i2n\theta'}d\theta'
\eqs
$\beta_n$ is real since:
\bqs
\Im(\beta_n)=-\frac{\sinh(2\omega)}{|V'(\omega)|}\int_{0}^{\pi}\tilde{w}\circ\tanh^{-1}\Bigg(\frac{2\tanh(\omega)\sin(\theta')}{\sqrt{(1-\tanh(\omega)^2)^2+4\tanh(\omega)^2\sin(\theta')^2}}\Bigg)\sin(2n\theta')d\theta'=0.
\eqs
Hence,
\bqs
\beta_n=\Re(\beta_n)=-\alpha+\frac{\sinh(2\omega)}{|V'(\omega)|}\int_{0}^{\pi}\tilde{w}\circ\tanh^{-1}\Bigg(\frac{2\tanh(\omega)\sin(\theta')}{\sqrt{(1-\tanh(\omega)^2)^2+4\tanh(\omega)^2\sin(\theta')^2}}\Bigg)\cos(2n\theta')d\theta'
\eqs
Since $\tilde{w}\circ\tanh^{-1}(r)$ is a positive function of $r$, it follows that:
\bqs
\beta_n\leq\beta_0
\eqs
Stability of the hyperbolic stationary pulse requires that for all $n\geq 0$, $\beta_n<0$. This can be rewritten as:
\bqs
 \frac{\sinh(2\omega)}{|V'(\omega)|}\int_{0}^{\pi}\tilde{w}\circ\tanh^{-1}\Bigg(\frac{2\tanh(\omega)\sin(\theta')}{\sqrt{(1-\tanh(\omega)^2)^2+4\tanh(\omega)^2\sin(\theta')^2}}\Bigg)\cos(2n\theta')d\theta'<\alpha \quad n\geq0
\eqs
Using the fact that $\beta_n\leq\beta_0$ for all $n\geq1$, we obtain the reduced stability condition:
\bqs
\frac{\mathcal{W}_{0}(\omega)}{|V'(\omega)|}<\alpha
\eqs
where
\bqs
\mathcal{W}_{0}(\omega)\overset{def}{=}\sinh(2\omega)\int_{0}^{\pi}\tilde{w}\circ\tanh^{-1}\Bigg(\frac{2\tanh(\omega)\sin(\theta')}{\sqrt{(1-\tanh(\omega)^2)^2+4\tanh(\omega)^2\sin(\theta')^2}}\Bigg)d\theta'
\eqs

From \eqref{eq:radeq} we have:
\bqs
V'(\omega)=\frac{1}{\alpha}(-\mM_{r}(\omega)+I'(\omega))
\eqs
where
\bqs
\mM_{r}(\omega)\overset{def}{=}-\frac{\partial\mM}{\partial r}(\omega,\omega)=\frac{1}{64}\sinh(2\omega)^3\int_{\R}\widetilde{\mbW}(\lambda)(1+\lambda^2)\Phi_{\lambda}^{(1,1)}(\omega)\Phi_{\lambda}^{(1,1)}(\omega)\lambda \tanh(\frac{\pi}{2}\lambda)d\lambda
\eqs
We have previously established that $\mM_{r}(\omega)>0$ and $I'(\omega)$ is negative by definition. Hence, letting $\mathcal{D}(\omega)=|I'(\omega)|$, we have
\bqs
|V'(\omega)|=\frac{1}{\alpha}(\mM_{r}(\omega)+\mathcal{D}(\omega)).
\eqs
By substitution we obtain another form of the reduced stability condition:
\bq
\mathcal{D}(\omega)>\mathcal{W}_{0}(\omega)-\mM_{r}(\omega)
\label{eq:stabcond}
\eq
We also have:
\bqs
\mM'(\omega)=\frac{d}{d\omega}\mM(\omega,\omega)=\frac{\partial\mM}{\partial r}(\omega,\omega)+\frac{\partial\mM}{\partial \omega}(\omega,\omega)=\mW_0(\omega)-\mM_r(\omega),
\eqs
and
\bqs
N'(\omega)=\mM'(\omega)+I'(\omega)=\mathcal{W}_{0}(\omega)-\mM_{r}(\omega)-\mathcal{D}(\omega),
\eqs
showing that the stability condition \eqref{eq:stabcond} is satisfied when $N'(\omega)<0$ and is not satisfied when $N'(\omega)>0$. 

\section{Numerical results}

The aim of this section is to numerically solve \eqref{eq:semihom} for different values of the parameters. This implies developing a numerical scheme that approaches the solution of our equation, and proving that this scheme effectively converges to the solution.\\
Since equation \eqref{eq:semihom} is defined on $\D$,  computing the solutions on the whole hyperbolic disk has same the complexity as computing the solutions of usual Euclidean neural field equations defined on $\R^2$. As most authors in the Euclidean case \cite{folias-bressloff:04,laing-troy:03,laing-troy-etal:02,owen-laing-etal:07}, we reduce the domain of integration to a compact region of the hyperbolic disk. Practically, we  work on the Euclidean ball of radius $a=0.5$ and center $0$. Note that a Euclidean ball centered at the origin is also a centered hyperbolic ball, their radii being different.\\
We have divided this section in four parts. The first part is dedicated to the study of the discretization scheme of equation \eqref{eq:semihom}. In the following three parts, we study the solutions of different connectivity functions: exponential function \ref{subsection:exponential}, Gabor function \ref{subsection:gabor} and a difference of Gaussians function \ref{subsection:gaussian}.

\subsection{Numerical schemes}

Let us consider the modified equation of \eqref{eq:semihom}:
\bq \left\{
    \begin{array}{ll}
\partial_t V(z,t)=-\alpha V(z,t)+\int_{B(0,a)} W(z,z')S(V(z',t))\text{dm}(z')+I(z,t)\quad t\in J\\
V(z,0)=V_0(z)
\end{array}
\right.
\label{eq:modifhom}
\eq
We assume that the connectivity function satisfies the conditions (\textbf{C1})-(\textbf{C2}). Moreover we express $z$ in (Euclidean) polar coordinates such that $z=re^{i\theta}$, $V(z,t)=V(r,\theta,t)$ and $W(z,z')=W(r,\theta,r',\theta')$. The integral in equation \eqref{eq:modifhom} is then:
\bqs
\int_{B(0,a)} W(z,z')S(V(z',t))\text{dm}(z')=\int_0^a \int_0^{2\pi} W(r,\theta,r',\theta')S(V(r',\theta',t))\frac{r'dr'd\theta'}{(1-r'^2)^2}
\eqs
We define $\mR$ to be the rectangle $\mR\overset{def}{=}[0,a]\times[0,2\pi]$.
\subsubsection{Discretization scheme}
We discretize $\mR$ in order to turn \eqref{eq:modifhom} into a finite number of equations. For this purpose we introduce $h_1=\frac{a}{N}$, $N\in\mathbb{N}^{*}=\mathbb{N}\backslash\{0\}$ and $h_2=\frac{2\pi}{M}$, $M\in\mathbb{N}^{*}$,
\bqs
\forall i\in \llbracket 1,N+1 \rrbracket\quad r_i=(i-1)h_1,
\eqs
\bqs
\forall j\in \llbracket 1,M+1 \rrbracket\quad \theta_j=(j-1)h_2,
\eqs
and obtain the $(N+1)(M+1)$ equations:
\bqs
\frac{dV}{dt}(r_i,\theta_j,t)=-\alpha V(r_i,\theta_j,t)+\int_{\mR}W(r_i,\theta_j,r',\theta')S(V(r',\theta',t))\frac{r'dr'd\theta'}{(1-r'^2)^2}+I(r_i,\theta_j,t)
\eqs
which define the discretization of \eqref{eq:modifhom}:

\bq \left\{
    \begin{array}{ll}
\frac{d\,\tilde{V}}{d\,t}(t)=-\alpha \tilde{V}(t)+\mbW\cdot S(\tilde{V})(t)+\tilde{I}(t)\quad t\in J\\
\tilde{V}(0)=\tilde{V}_0
\end{array}
\right.
\label{eq:discrhom}
\eq
where $\tilde{V}(t)\in\mM_{N+1,M+1}(\R)$ \footnote{$\mM_{n,p}(\R)$ is the space of the matrices of size $n\times p$ with real coefficients.}, $\tilde{V}(t)_{i,j}=V(r_i,\theta_j,t)$. Similar definitions apply to $\tilde{I}$ and $\tilde{V}_0$. Moreover:
\bqs
\mbW\cdot S(\tilde{V})(t)_{i,j}=\int_{\mR}W(r_i,\theta_j,r',\theta')S(V(r',\theta',t))\frac{r'dr'd\theta'}{(1-r'^2)^2}
\eqs
It remains to discretize the integral term. For this as in \cite{faye-faugeras:09}, we use the rectangular rule for the quadrature so that for all $(r,\theta)\in\mR$ we have:
\bqs
\int_0^a \int_0^{2\pi} W(r,\theta,r',\theta')S(V(r',\theta',t))\frac{r'dr'd\theta'}{(1-r'^2)^2}\cong h_1h_2\sum_{k=1}^{N+1}\sum_{l=1}^{M+1}W(r,\theta,r_k,\theta_l)S(V(r_k,\theta_l,t))\frac{r_k}{(1-r_k^2)^2}
\eqs
We end up with the following numerical scheme, where $\mcV_{i,j}(t)$ (resp. $\mathcal{I}_{i,j}(t)$) is an approximation of $\tilde{V}_{i,j}(t)$ (resp. $\tilde{I}_{i,j}$), $\forall (i,j)\in\llbracket 1,N+1\rrbracket \times \llbracket 1,M+1\rrbracket$:
\bqs
\frac{d\mcV_{i,j}}{dt}(t)=-\alpha \mcV_{i,j}(t)+h_1h_2\sum_{k=1}^{N+1}\sum_{l=1}^{M+1}\widetilde{W}_{k,l}^{i,j}S(\mcV_{k,l})(t)+\mathcal{I}_{i,j}(t)
\eqs
with $\widetilde{W}_{k,l}^{i,j}\overset{def}{=}W(r_i,\theta_j,r_k,\theta_l)\frac{r_k}{(1-r_k^2)^2}$.
\subsubsection{Discussion}
We discuss the error induced by the rectangular rule for the quadrature. Let $f$ be a function which is $\mathcal{C}^2$ on a rectangular domain $[a,b]\times[c,d]$. If we denote by $E_f$ this error, then $|E_{f}|\leq \frac{(b-a)^2(d-c)^2}{4mn}\| f\|_{\mC^2}$ where $m$ and $n$ are the number of subintervals used and $\| f\|_{\mC^2}=\sum_{|\alpha|\leq 2}\sup_{[a,b]\times[c,d]}|\partial^{\alpha}f|$ where, as usual, $\alpha$ is a multi-index. As a consequence, if we want to control the error, we have to impose that the solution is, at least, $\mathcal{C}^2$ in space.\\
Four our numerical experiments we use the specific function $\text{ode}45$ of Matlab which is based on an explicit Runge-Kutta (4,5) formula (see \cite{bellen-zennaro:05} for more details on Rung-Kutta methods).\\
We can also establish a proof of the convergence of the numerical scheme which is exactly the same as in \cite{faye-faugeras:09} excepted that we use the theorem of continuous dependence of the solution for ordinary differential equations.

\subsection{Purely excitatory exponential connectivity function}\label{subsection:exponential}
In this subsection, we give some numerical solutions of \eqref{eq:semihom} in the case where the connectivity function is an exponential function, $w(x)=e^{-\frac{|x|}{b}}$, with $b$ a positive parameter. Only excitation is present in this case. In all the experiments we set $\alpha=0.1$ and $S(x)=\frac{1}{1+e^{-\mu x}}$ with $\mu=10$.
\paragraph{Constant input} We fix the external input $I(z)$ to be of the form:
\bqs
I(z)=\mathcal{I}e^{-\dfrac{d_2(z,0)^2}{\sigma^2}}
\eqs
In all experiments we set $\mathcal{I}=0.1$ and $\sigma=0.05$, this means that the input has a sharp profile centered at $0$.
\begin{figure}[htbp]
 \centering
\subfigure[$b=1$]{
\label{fig:expdb1}
\includegraphics[width=0.4\textwidth,bb=0 0 448 336]{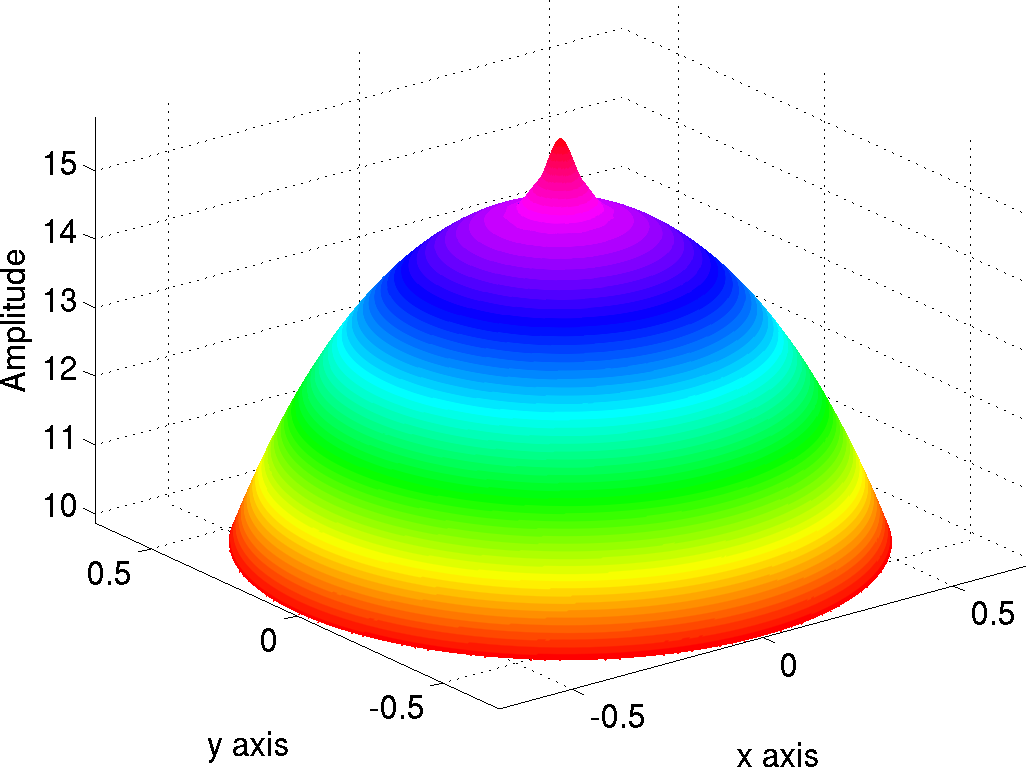}}
\hspace{.3in}
\subfigure[$b=0.2$]{
\label{fig:expdb02}
\includegraphics[width=0.4\textwidth,bb=0 0 448 336]{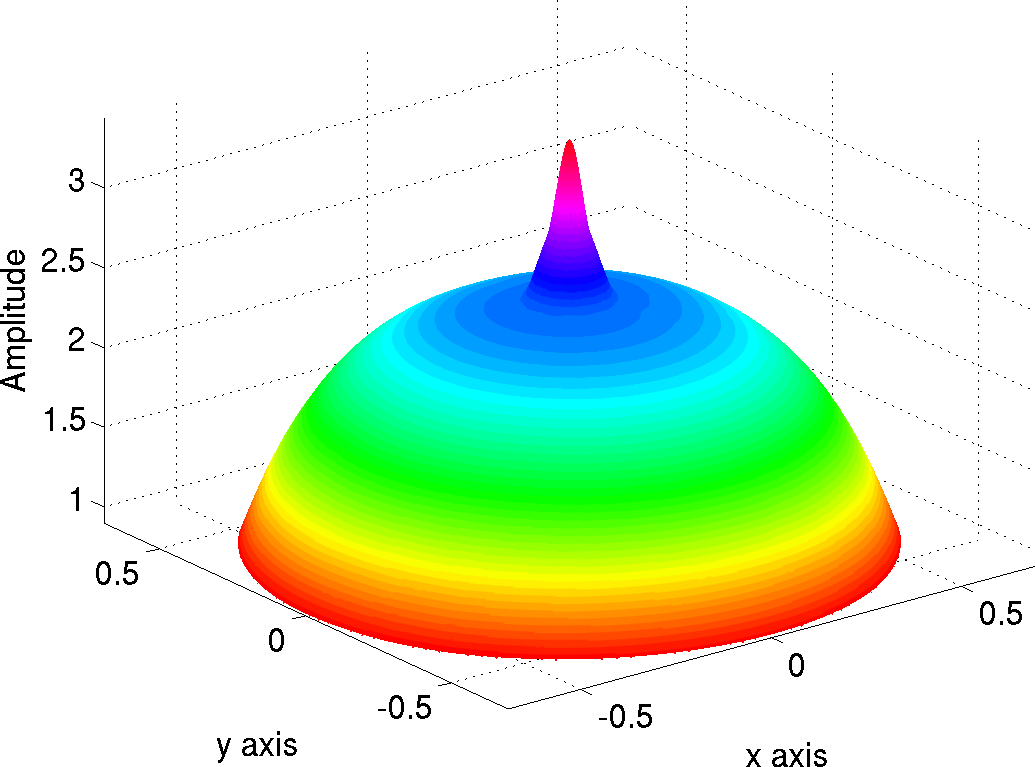}}\\
\subfigure[$b=0.1$]{
\label{fig:expdb01}
\includegraphics[width=0.4\textwidth,bb=0 0 448 336]{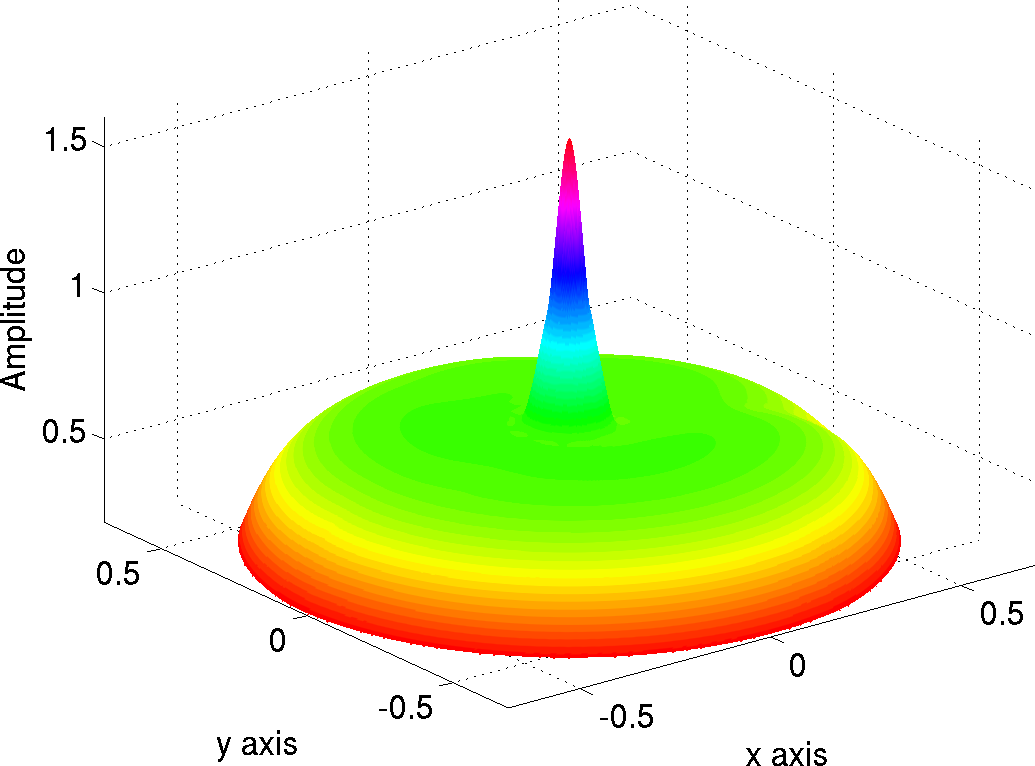}}
\caption{Plots of the solution of equation \eqref{eq:semihom} at $T=2500$ for the values $\mu=10,\alpha=0.1$ and for decreasing values of the width $b$ of the connectivity, see text.}
\label{fig:expd}
\end{figure}
We show in figure \ref{fig:expd} plots of the solution at time $T=2500$ for three different values of the width $b$ of the exponential function. When $b=1$, the whole network is hightly excited, see figure \ref{fig:expdb1}. When $b$ changes from $1$ to $0.1$ the amplitude of the solution decreases, and the area of high excitation becomes concentrated around the external input. 
\paragraph{Variable input} In this paragraph, we allow the external current to depend upon the time variable. We have:
\bqs
I(z,t)=\mathcal{I}e^{-\dfrac{d_2(z,z_0(t))^2}{\sigma^2}}
\eqs
where $z_0(t)=r_0 e^{i\Omega_0 t}$. This is a bump rotating with angular velocity $\Omega_0$ around the circle of radius $r_0$ centered at the origin. In our numerical experiments we set $r_0=0.4$, $\Omega_0=0.01$, $\mathcal{I}=0.1$ and $\sigma=0.05$. We plot in figure \ref{fig:expvariable} the solution at different times $T=100,150,200,250$.
\begin{figure}
 \centering
 \includegraphics[width=0.6\textwidth,bb=0 0 467 385]{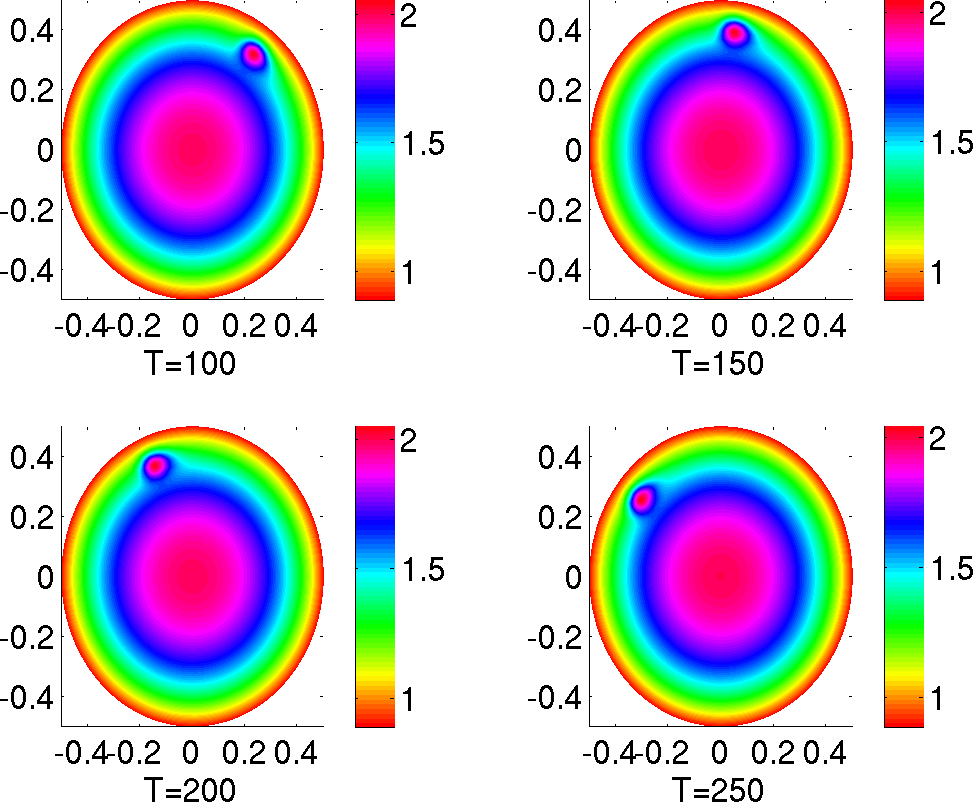}
 \caption{Plots of the solution of equation (\ref{eq:semihom}) in the case of an exponential connectivity function with $b=0.1$ at different times with a time-dependent input, see text.}
 \label{fig:expvariable}
\end{figure}

\paragraph{High gain limit} We consider the high gain limit $\mu\rightarrow\infty$ of the sigmoid function and we propose to illustrate section \ref{section:bump} with a numerical simulation. We set $\alpha=1$, $\kappa=0.04$, $\omega=0.18$. We fix the input to be of the form:
\bqs
I(z)=\mathcal{I}e^{-\dfrac{d_2(z,0)^2}{\sigma^2}}
\eqs
with $\mathcal{I}=0.04$ and $\sigma=0.05$. Then the condition of existence of a stationary pulse \eqref{eq:excond} is satisfied, see figure \ref{fig:Nomega}. We plot a bump solution according to \eqref{eq:excond} in figure \ref{fig:bump}.\\

\begin{figure}[htbp]
 \centering
\subfigure[$b=0.2$]{
\label{fig:bump1}
\includegraphics[width=0.4\textwidth,bb=0 0 521 379]{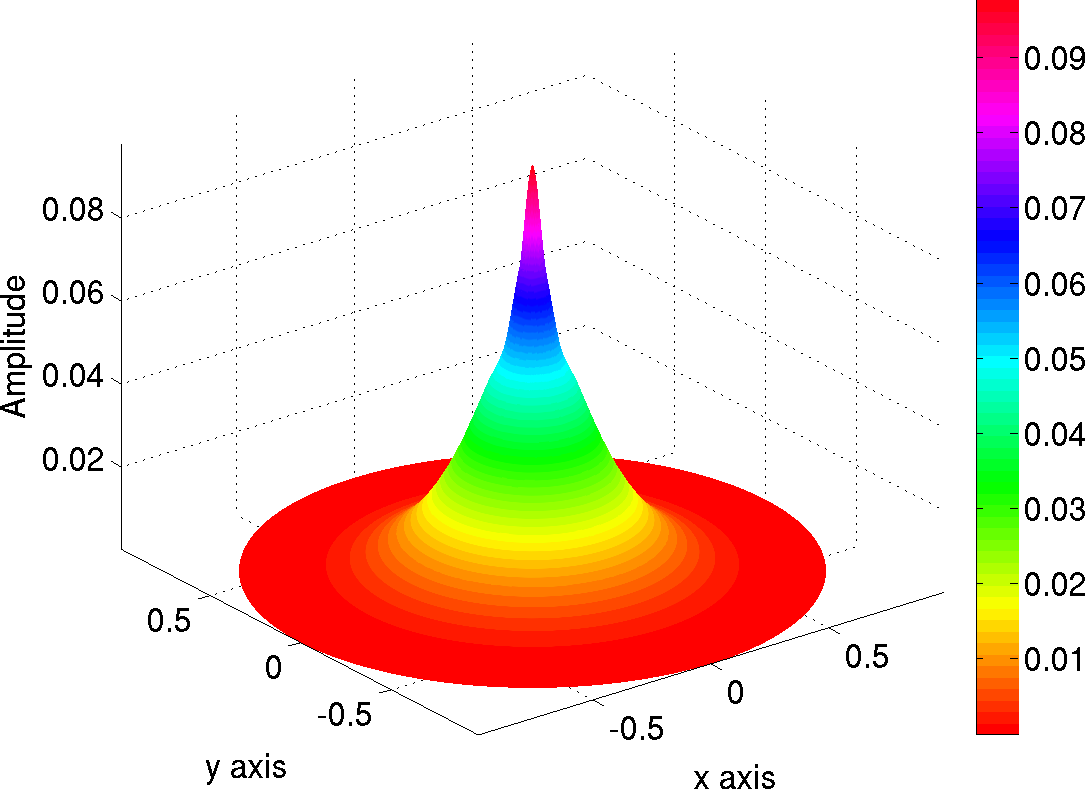}}
\hspace{.3in}
\subfigure[$b=0.2$]{
\label{fig:bump2}
\includegraphics[width=0.4\textwidth,bb=0 0 521 379]{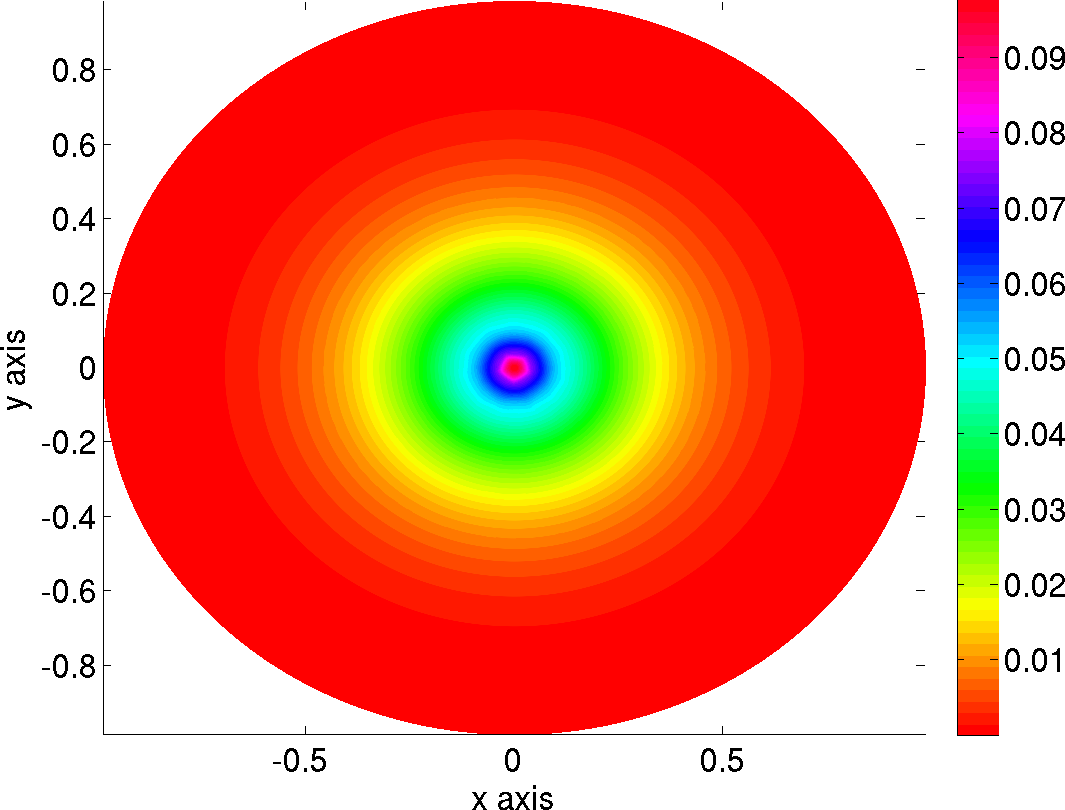}}\\
\caption{Plots of a bump solution of equation \eqref{eq:radeq} for the values $\alpha=1$, $\kappa=0.04$, $\omega=0.18$ and for $b=0.2$ for the width of the connectivity, see text.}
\label{fig:bump}
\end{figure}

\subsection{Excitatory and inhibitory connectivity function}\label{subsection:gabor}
We give some numerical solutions of \eqref{eq:semihom} in the case where the connectivity function is a Gabor function. In all the experiments we set $\alpha=0.1$ and $S(x)=\frac{1}{1+e^{-\mu x}}$ with $\mu=10$. 
The Gabor function is given by $G(x)=\frac{1}{\sqrt{b}}\left(1-2\frac{x^2}{b^2}\right)e^{-\frac{x^2}{b}}$, where $b$ is a positive parameter, takes positive values near the origin and negative values further away.

The external input $I(z,t)$ is equal to zero. In figure \ref{fig:gabor}, we show several solutions of equation \eqref{eq:semihom} at $T=2500$ for decreasing values of the width $b$ of the connectivity. We see the emergence of multi-bump solutions. In figure \ref{fig:b04}, there is one bump centered at $O$. When decreasing the width $b$ of the connectivity, for the range of values $b\in[0.2;0.35]$, the solutions present an increasing number of areas of high level of activity. For example, in figure \ref{fig:b02}, there exist five such areas. Figure \ref{fig:b01} shows the appearance of a second crown of localized patterns centered at $0$, and figure \ref{fig:b007} shows three such crowns. Note that all the solutions display interesting symmetries.
\begin{figure}[htbp]
 \centering
\subfigure[$b=0.4$]{
\label{fig:b04}
\includegraphics[width=0.4\textwidth,bb=0 0 495 388]{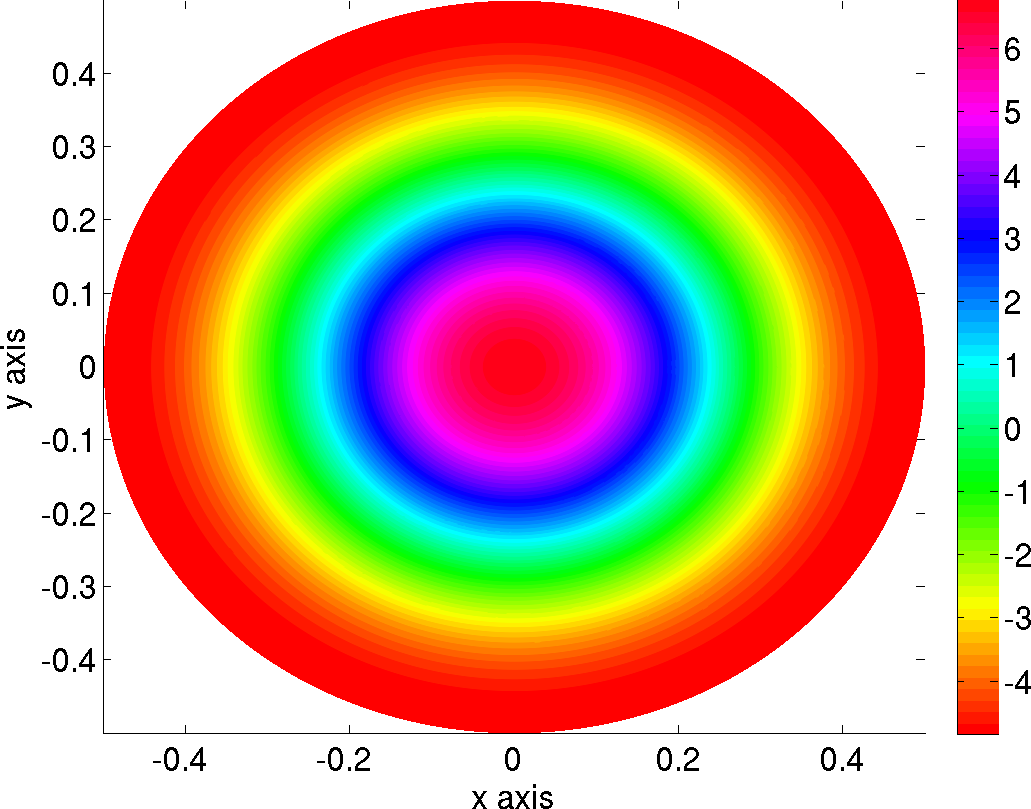}}
\hspace{.3in}
\subfigure[$b=0.35$]{
\label{fig:b035}
\includegraphics[width=0.4\textwidth,bb=0 0 495 388]{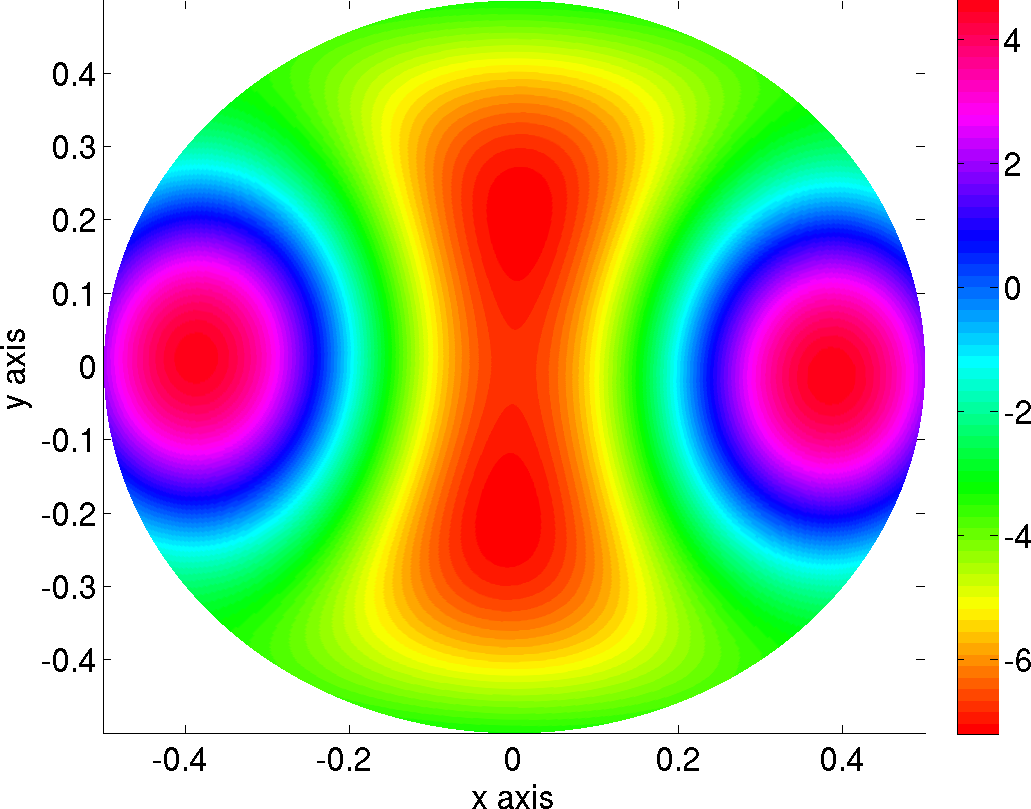}}\\
\subfigure[$b=0.3$]{
\label{fig:b03}
\includegraphics[width=0.4\textwidth,bb=0 0 495 388]{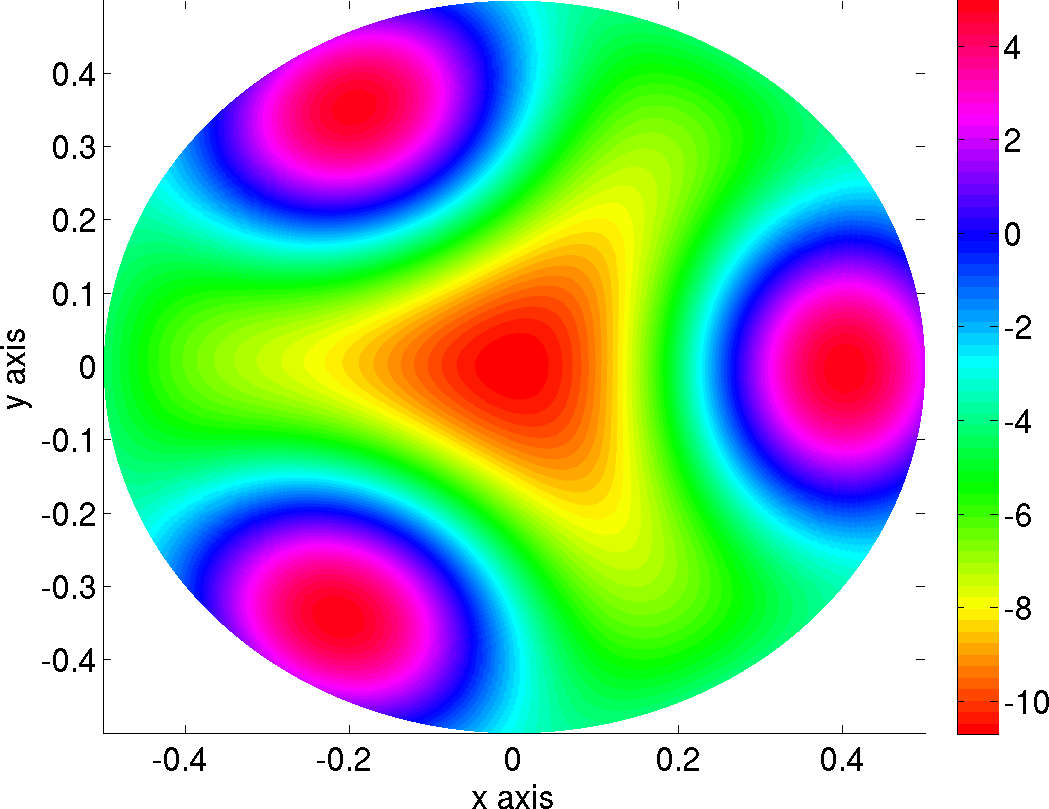}}
\hspace{.3in}
\subfigure[$b=0.2$]{
\label{fig:b02}
\includegraphics[width=0.4\textwidth,bb=0 0 495 388]{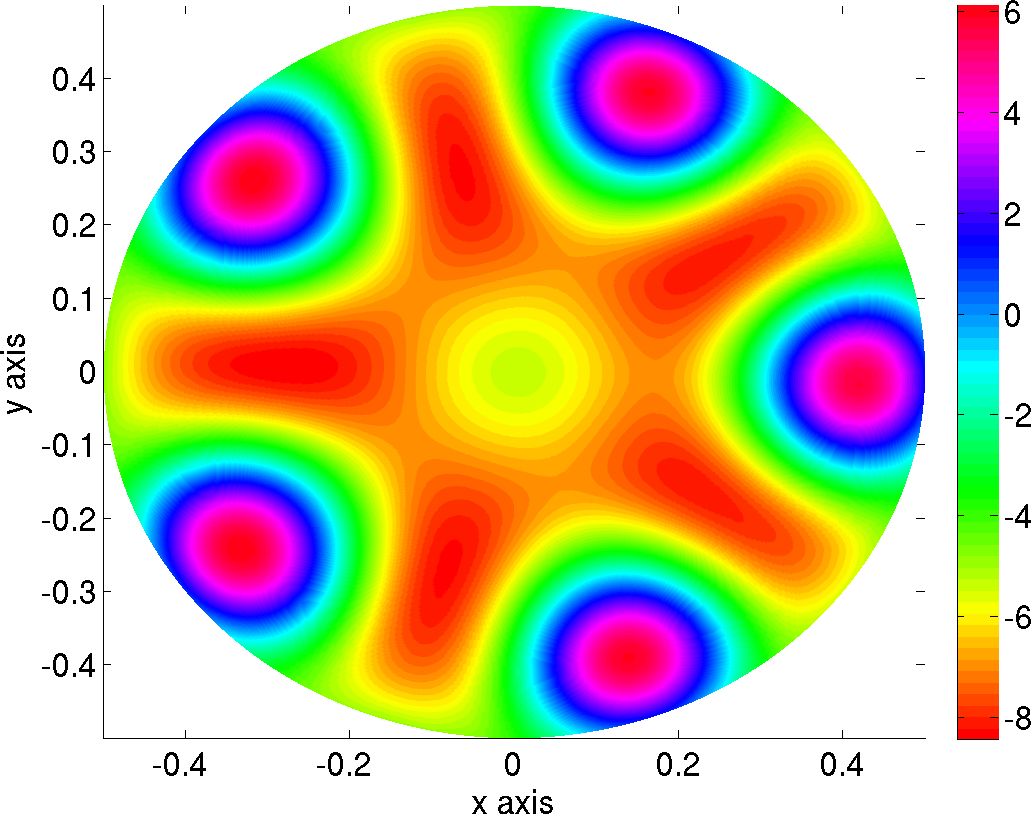}}\\
\subfigure[$b=0.1$]{
\label{fig:b01}
\includegraphics[width=0.4\textwidth,bb=0 0 495 388]{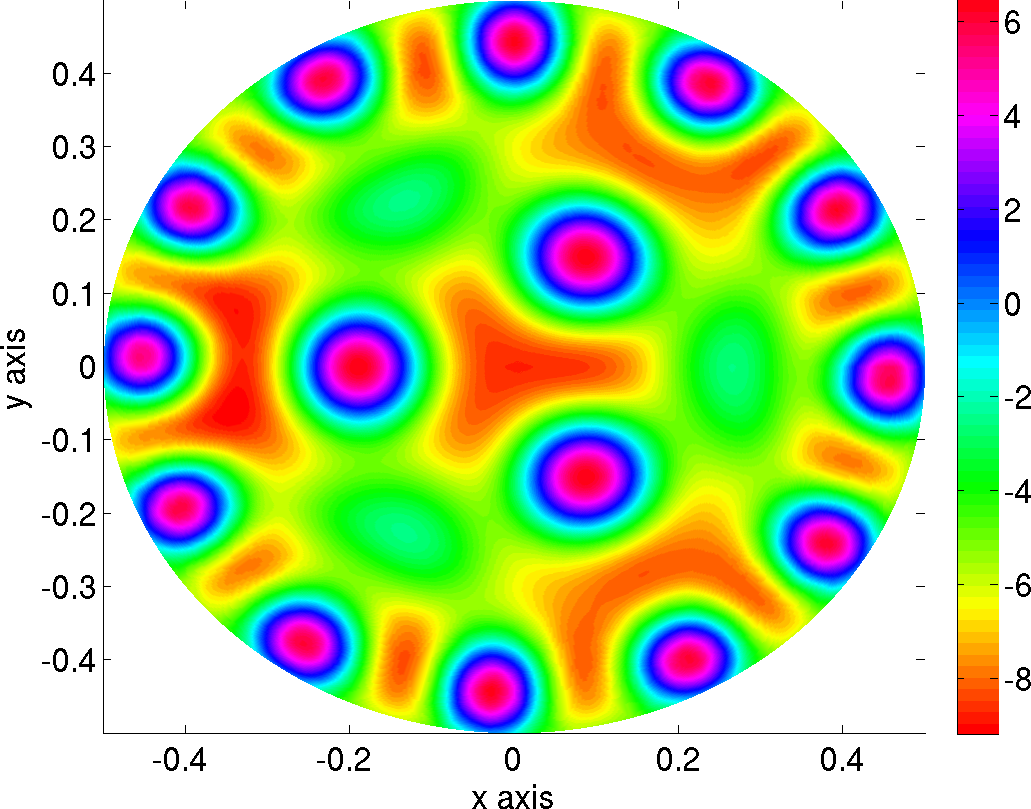}}
\hspace{.3in}
\subfigure[$b=0.07$]{
\label{fig:b007}
\includegraphics[width=0.4\textwidth,bb=0 0 495 388]{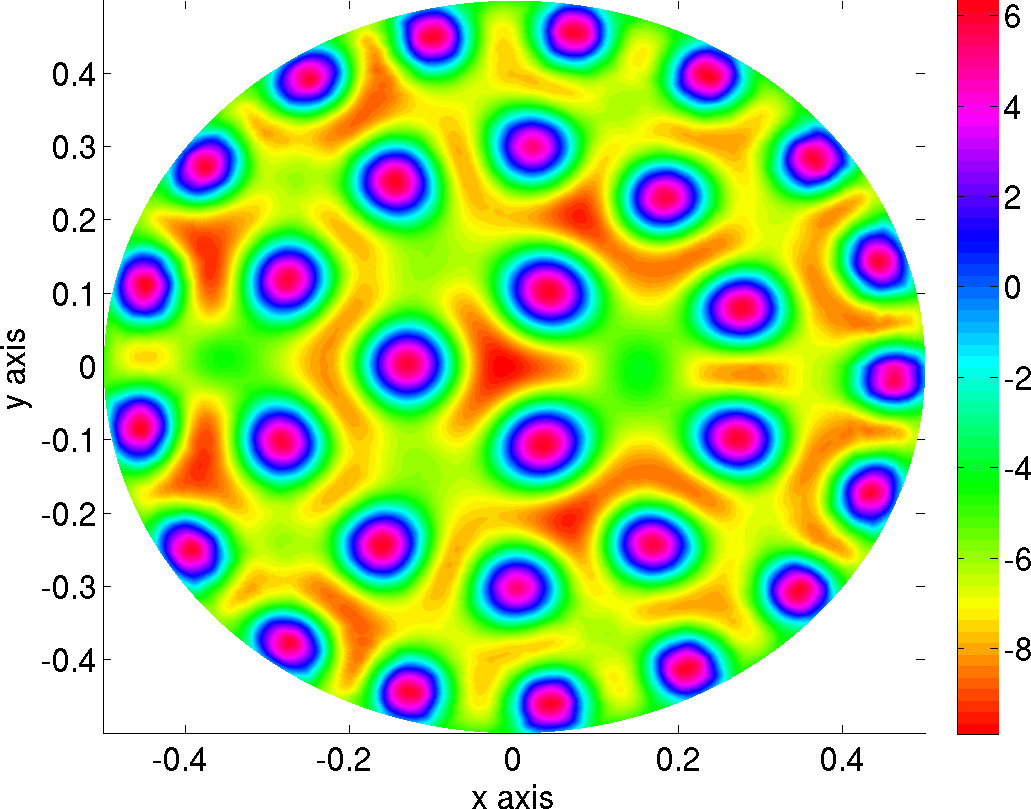}}\\

\caption{Plots of the solutions of equation \eqref{eq:semihom} in the case of a Gabor connectivity function at time $T=2500$ for the values $\mu=10,\alpha=0.1$ and for decreasing values of the width $b$ of the connectivity, see text.}
\label{fig:gabor}
\end{figure}

We slightly change the shape of the connectivity function by choosing the difference of two Gaussians $w(x)=\frac{1}{\sqrt{2\pi\sigma_1^2}}e^{-\frac{x^2}{\sigma_1^2}}-
\frac{A}{\sqrt{2\pi\sigma_2^2}}e^{-\frac{x^2}{\sigma_2^2}}$. In all the experiments, we take $\sigma_1=0.1,\sigma_2=0.2$ and $A=1$.\\
As shown in figure \ref{fig:gabormexhat} the two connectivity functions, Gabor and difference of Gaussians, feature an excitatory center and an inibitory surround.
\begin{figure}
 \centering
 \includegraphics[width=0.5\textwidth,bb=0 0 484 394]{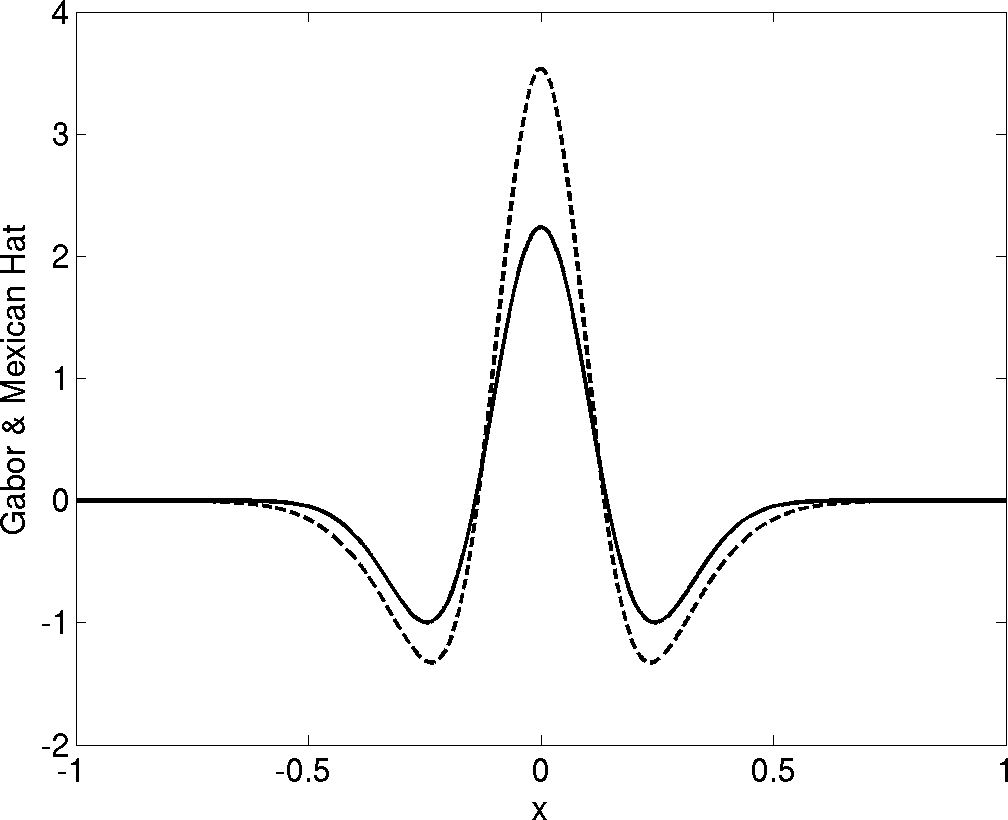}
 \caption{Plot of the Gabor function $G$ with $b=0.2$ and in dashed lines the difference of Gaussian functions $w$ with $\sigma_1=0.1,\sigma_2=0.2$ and $A=1$, see text. They both feature an excitatory center and an inhibitory surround.}
 \label{fig:gabormexhat}
\end{figure}
We illustrate the behaviour of the solutions when increasing the slope $\mu$ of the sigmoid. We set the sigmoid $S(x)=\frac{1}{1+e^{-\mu x}}-\frac{1}{2}$ so that it is equal to 0 at the origin and we choose the external input equal to zero, $I(z,t)=0$. In this case the constant function equal to 0 is a solution of \eqref{eq:semihom}.

\begin{figure}[htbp]
 \centering
\subfigure[$\mu=1$]{
\label{fig:mu1}
\includegraphics[width=0.4\textwidth,bb=0 0 495 388]{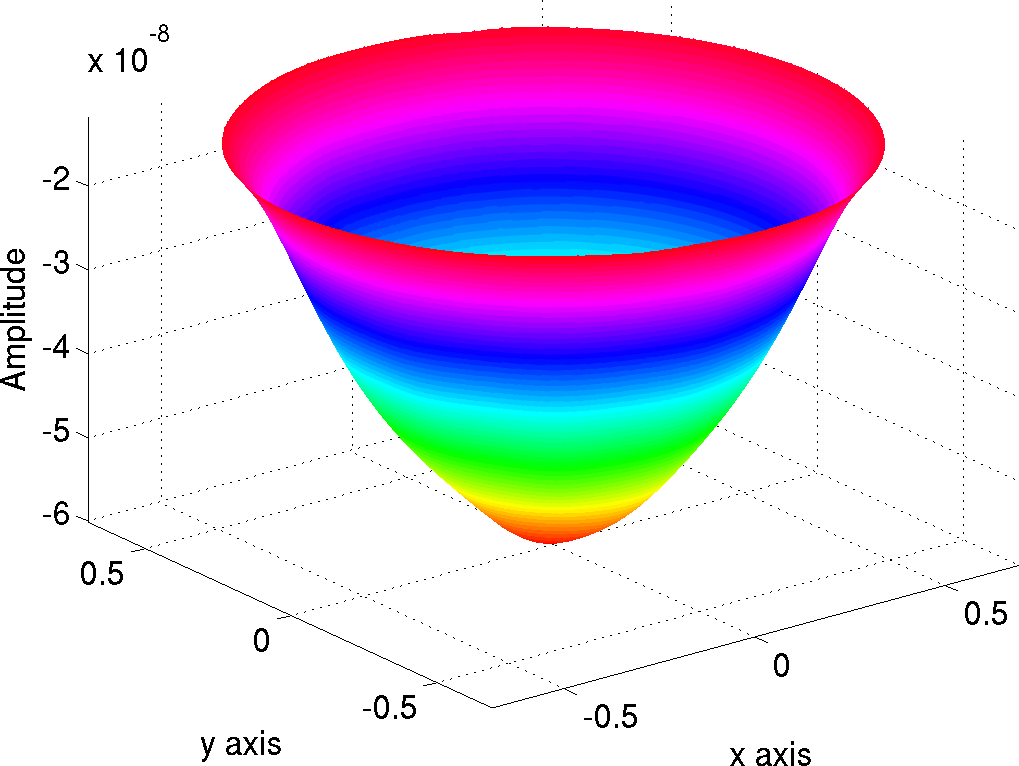}}
\hspace{.3in}
\subfigure[$\mu=3$]{
\label{fig:mu3}
\includegraphics[width=0.4\textwidth,bb=0 0 495 388]{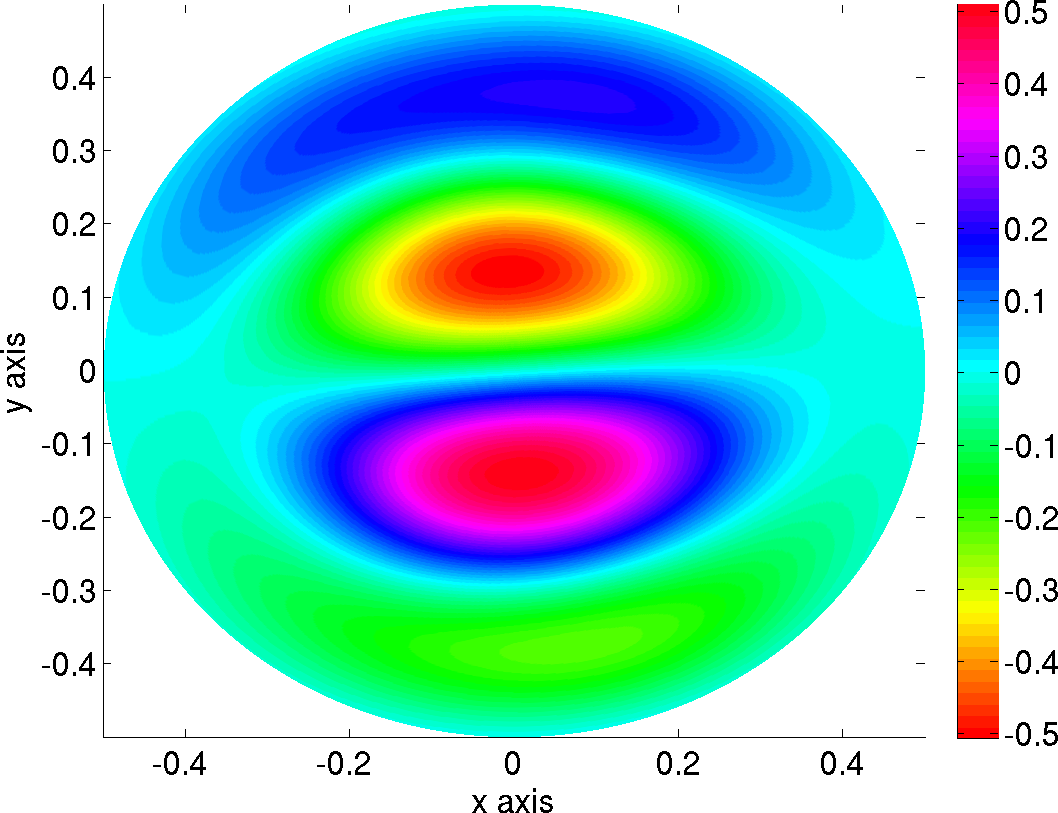}}\\
\subfigure[$\mu=5$]{
\label{fig:mu5}
\includegraphics[width=0.4\textwidth,bb=0 0 495 388]{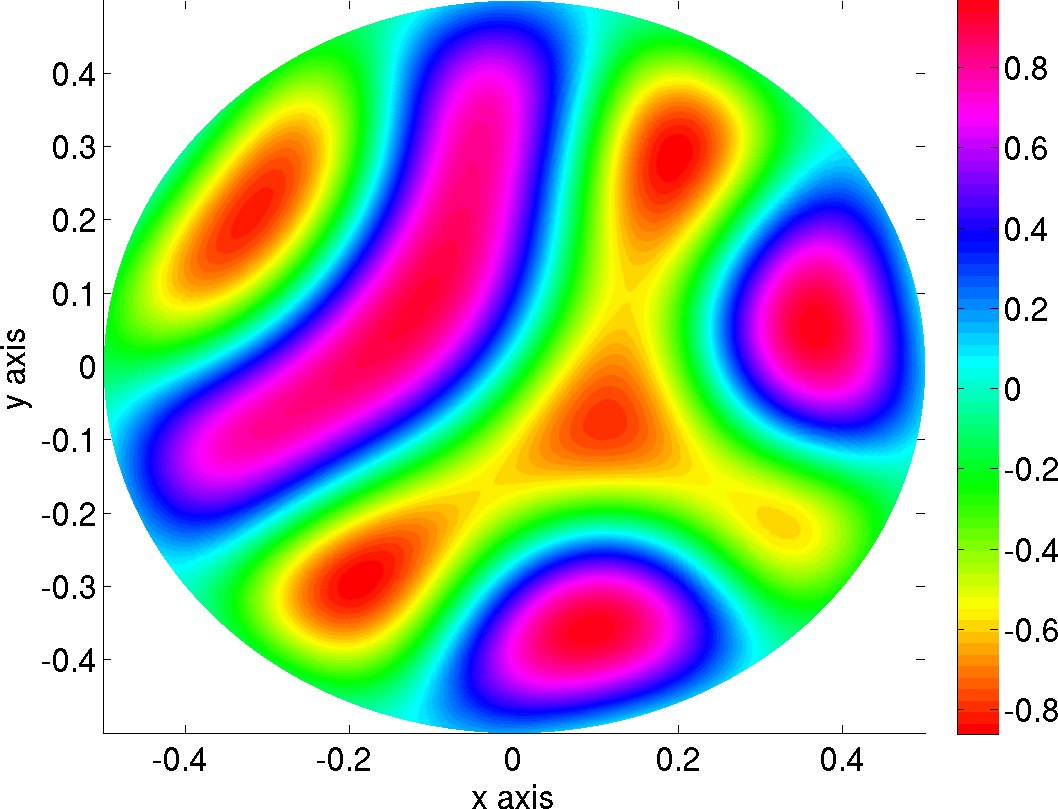}}
\hspace{.3in}
\subfigure[$\mu=10$]{
\label{fig:mu10}
\includegraphics[width=0.4\textwidth,bb=0 0 495 388]{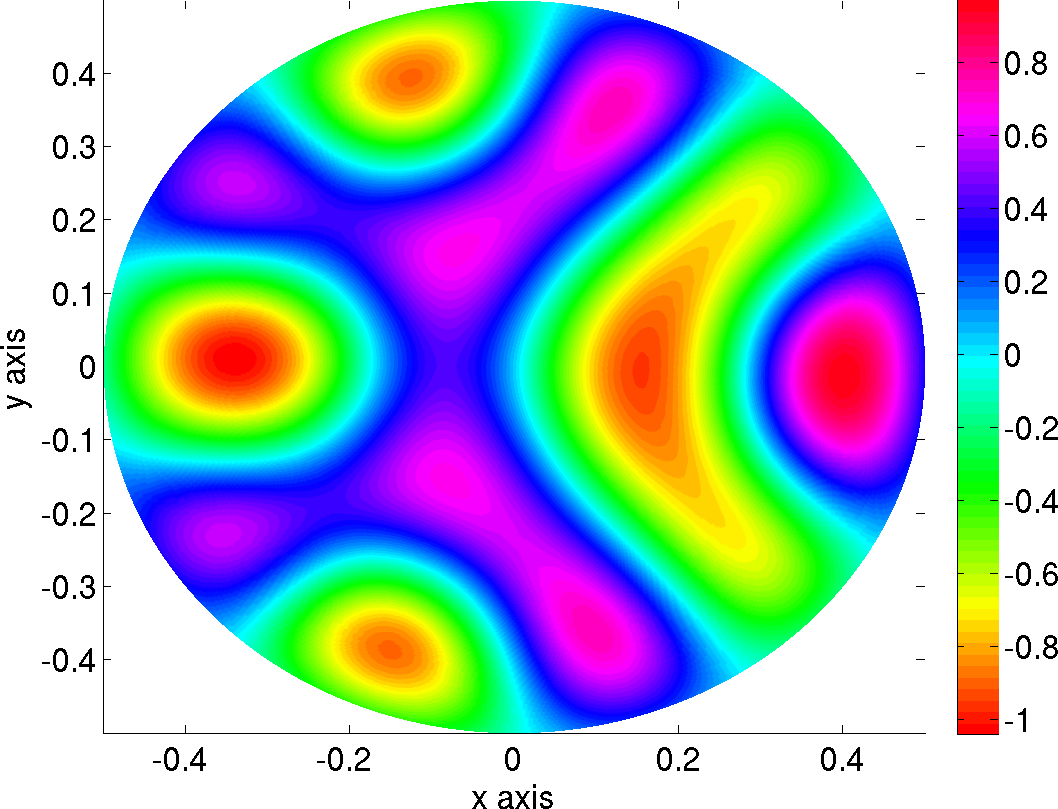}}\\
\subfigure[$\mu=20$]{
\label{fig:mu20}
\includegraphics[width=0.4\textwidth,bb=0 0 495 388]{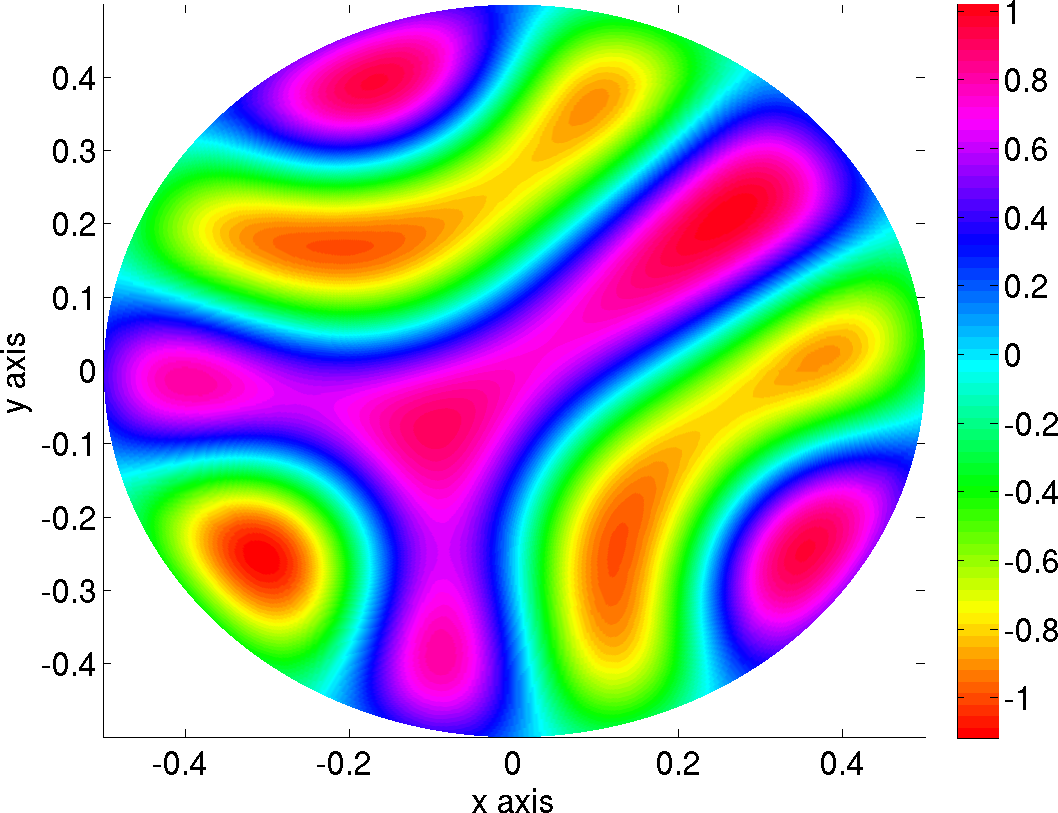}}
\hspace{.3in}
\subfigure[$\mu=30$]{
\label{fig:mu30}
\includegraphics[width=0.4\textwidth,bb=0 0 495 388]{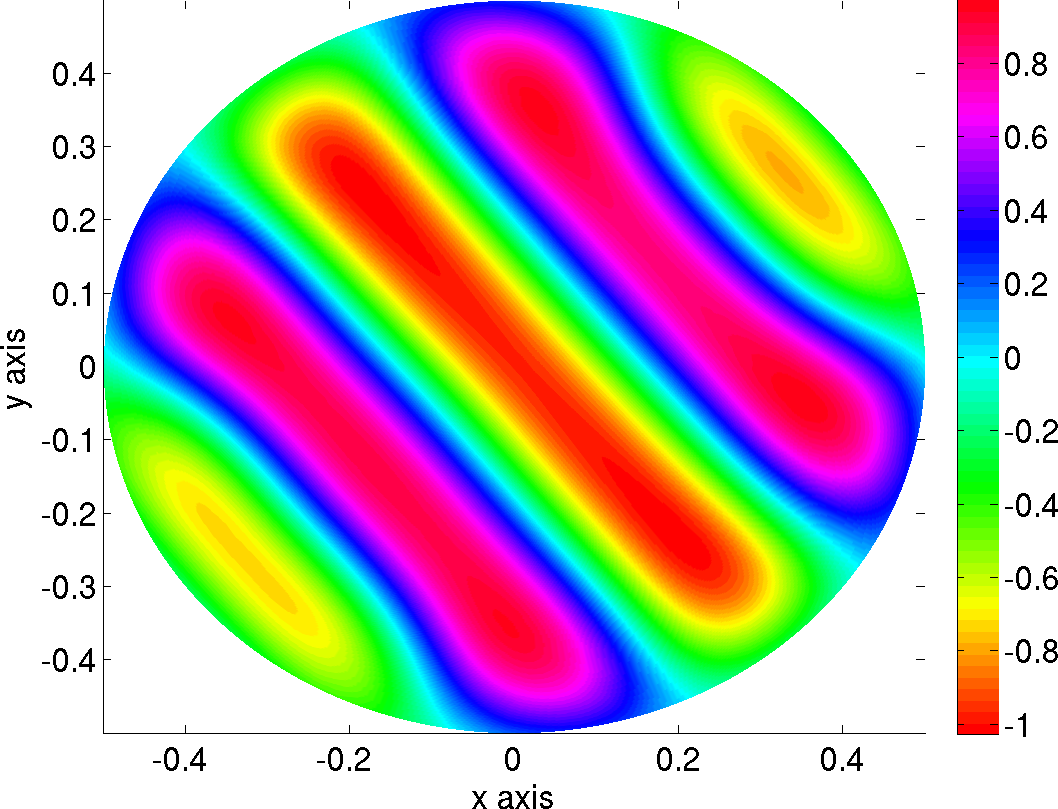}}\\
\caption{Plots of the solutions of equation \eqref{eq:semihom} in the case where the connectivity function is the difference of two Gaussians at time $T=2500$ for $\alpha=0.1$ and for increasing values of the slope $\mu$ of the sigmoid, see text.}
\label{fig:diffgauss}
\end{figure}
For small values of the slope $\mu$, the dynamics of the solution is trivial: every solution asymptotically converges to the null solution, as shown in figure \ref{fig:mu1}. When increasing $\mu$, the stability bound, found in subsection \ref{subsection:stability} is no longer satisfied and the null solution may no longer be stable. In effect this solution may bifurcate to other, more interesting solutions. We plot in figures \ref{fig:mu3},\ref{fig:mu5},\ref{fig:mu10},\ref{fig:mu20} and \ref{fig:mu30}, some solutions at $T=2500$ for different values of $\mu$. We can see exotic patterns which feature some interesting symmetries. The formal study of these bifurcated solutions is left for future work.

\section{Conclusion}
\label{section:conclusion}
We studied the existence and uniqueness of a solution of the equation of a smooth neural field model of structure tensors to model the representation and processing of texture and edges in the visual area V1. We also detailed the analysis of the stationary solutions of this nonlinear integro-differential equation. In both cases functional analysis and the theory of ordinary differential equations have allowed us to introduce the framework in which our equations are well-posed and to begin to characterize their solutions. This is interesting and important in itself but should also be useful for future investigations such as bifurcation analysis.\\
We have completed our study by constructing and analysing spatially localised bumps in the high-gain limit of the sigmoid function. It is true that networks with Heaviside nonlinearities are not very realistic from the neurobiological perspective and lead to difficult mathematical considerations. However, taking the high-gain limit is instructive since it allows the explicit construction of stationary solutions which is impossible with sigmoidal nonlinearities. We constructed what we called a hyperbolic radially symmetric stationary-pulse and presented a linear stability analysis adapted from \cite{folias-bressloff:04}.\\
Finally, we illustrated our theoretical results with numerical simulations based on rigorously defined numerical schemes. We hope that our numerical experiments will lead to new and exciting investigations such as a thorough study of the bifurcations of the solutions of our equations with respect to such parameters as the slope of the sigmoid and the width of the connectivity function.\\

\noindent
{\bf Acknowledgments}\\ 
This work was partially funded by the ERC advanced grant NerVi number 227747.

\appendix

\section{Volume element in structure tensor space}\label{appendix:volume}

Let $\T$ be a structure tensor
\[
 \T=\left[
	\begin{array}{cc}
	 x_1 & x_3\\
	 x_3  & x_2
	\end{array}
\right],
\]
$\Delta^2$ its determinant, $\Delta \geq 0$. $\T$ can be written
\[
 \T=\Delta \tT,
\]
where $\tT$ has determinant 1. Let $z=z_1+iz_2$ be the complex number representation of $\tT$ in the
Poincaré disk $\D$.
In this part of the appendix, we present a simple form for the volume element in full structure tensor space, when parametrized as $(\Delta,z)$. 

\begin{prop}\label{prop:volel}
 The volume element in $(\Delta,z_1,z_2)$ coordinates is 
\bq\label{eq:volume}
 dV=8\sqrt{2}\,\frac{d\Delta}{\Delta}\, \frac{dz_1\,dz_2}{(1-|z|^2)^2}
\eq
\end{prop}

\begin{proof}
In order to compute the volume element in $(\Delta,z_1,z_2)$ space, we need to express the metric $g_\T$ in these coordinates.
This is obtained from the inner product in the tangent space $T_\T$ at point $\T$ of ${\rm SDP}(2)$. The tangent space is the
set ${\rm S}(2)$ of symmetric matrices and the inner product is defined by:
\[
 g_\T( A, B )={\rm tr}(\T^{-1} A \T^{-1} B),\quad A,\,B \in {\rm S}(2),
\]
We note that $g_\T( A, B )=g_{\tT}(A,B)/\Delta^2$.
We note $g$ instead of $g_{\tT}$.
A basis of $T_\T$ (or $T_{\tT}$ for that matter) is given by:
\[
 \frac{\partial}{\partial x_1}=
	\left[
	\begin{array}{cc}
	 1 & 0\\
	 0  & 0
	\end{array}
	\right] \quad 
\frac{\partial}{\partial x_2}=
	\left[
	\begin{array}{cc}
	 0 & 0\\
	 0  & 1
	\end{array}
	\right] \quad
\frac{\partial}{\partial x_3}=
	\left[
	\begin{array}{cc}
	 0 & 1\\
	 1  & 0
	\end{array}
\right],
\]
and the metric is given by:
\[
 g_{ij}=g_{\tT}(\frac{\partial}{\partial x_i},\frac{\partial}{\partial x_j}),\,i,j=1,2,3
\]
The determinant $G_\T$ of $g_\T$ is equal to $G/\Delta^6$, where $G$ is the determinant of $g=g_{\tT}$. $G$ is found to be equal to 2.
The volume element is thus:
\[
 dV=\frac{\sqrt{2}}{\Delta^3}\,dx_1\,dx_2\,dx_3
\]
We then use the relations:
\[
 x_1=\Delta \tilde x_1 \quad  x_2=\Delta \tilde x_2 \quad x_3=\Delta \tilde x_3,
\]
where $\tilde x_i$, $i=1,2,3$ is given by:
\[
  \left\{
\begin{array}{lcl}
 \tilde x_1 & = & \frac{(1+z_1)^2+z_2^2}{1-z_1^2-z_2^2}\\
 \tilde x_2 & = & \frac{(1-z_1)^2+z_2^2}{1-z_1^2-z_2^2}\\
\tilde x_3 & = & \frac{2z_2}{1-z_1^2-z_2^2}
\end{array}
\right.,
\]
The determinant of the Jacobian of the transformation $(x_1,x_2,x_3) \to (\Delta,z_1,z_2)$ is found to be equal to:
\[
 -\frac{8\Delta^2}{(1-|z|^2)^2}
\]
Hence, the volume element in $(\Delta,z_1,z_2)$ coordinates is 
\bqs
 dV=8\sqrt{2}\,\frac{d\Delta}{\Delta}\, \frac{dz_1\,dz_2}{(1-|z|^2)^2}
\eqs

\end{proof}

\section{Global existence of solutions}

\begin{thm}\label{thm:cauchylipschitz}
 Let $\mO$ be an open connected set of a real Banach space $\mF$ and $J$ be an open interval of $\R$. We consider the initial value problem:
\bq \left\{ \begin{array}{lcl}
\mV'(t)&=&f(t,\mV(t)),\\
\mV(t_0)&=&\mV_0
\end{array}
\right.
\label{eq:cauchylipschitz}
\eq
We suppose that $f\in\mathcal{C}\left(J\times \mO,\mF \right)$ and is locally Lipschitz with respect to its second argument. Then for all $(t_0,\mV_0)\in J\times\mO$, there exists $\tau>0$ and $\mV\in\mC^1\left( ]t_\tau,t_0+\tau[,\mO\right)$ unique solution of (\ref{eq:cauchylipschitz}).
\end{thm}

\begin{lem}
 Under hypotheses of theorem \ref{thm:cauchylipschitz}, if $\mV_1\in\mC^1(J_1,\mO)$ and $\mV_2\in\mC^1(J_2,\mO)$ are two solutions and if there exists $t_0\in J_1 \cap J_2$ such that $\mV_1(t_0)= \mV_2(t_0)$ then:
\bqs
\mV_1(t)= \mV_2(t) \text{ for all } t\in J_1 \cap J_2
\eqs
\end{lem}
This lemma shows the existence of a larger interval $J_0$ on which the initial value problem (\ref{eq:cauchylipschitz}) has a unique solution. This solution is called the maximal solution.

\begin{thm}\label{thm:bouts}
 Under hypotheses of theorem \ref{thm:cauchylipschitz}, let $\mV\in\mC^1(J_0,\mO)$ be a maximal solution. We note by $b$ the upper bound of $J$ and $\beta$ the upper bound of $J_0$. Then either $\beta=b$ or for all compact set $\mK\subset \mO$, there exists $\eta<\beta$ such that:
\bqs
\mV(t)\in\mO | \mK, \text{ for all } t\geq \eta \text{ with } t\in J_0 \cdot
\eqs
We have the same result with the lower bounds.
\end{thm}

\begin{thm}\label{thm:appendixglobal}
 We suppose $f\in\mC \left( J\times \mF,\mF \right)$ and is globally Lipschitz with respect to its second argument. Then for all $(t_0,\mV_0)\in J\times \mF$, there exists a unique $\mV\in\mC^1(J,\mF)$ solution of \eqref{eq:cauchylipschitz}.
\end{thm}

%

\section{Proof of lemma \ref{lemma:mexican}}\label{appendix:mexican}
In this section we prove the following lemma.
\begin{lem}
 When $W$ is a mexican hat function of $d_2(\T,\T')$ and independent of $t$, then:
\bqs
\overline{W}=\frac{\pi^{\frac{3}{2}}}{2}\left(\sigma_1e^{2\sigma_1^2}\textbf{erf}\left(\sqrt{2}\sigma_1\right)-A\sigma_2e^{2\sigma_2^2}\textbf{erf}\left(\sqrt{2}\sigma_2\right) \right)
\eqs
where $\textbf{erf}$ is the error function defined as:
\bqs
\textbf{erf}(x)=\frac{2}{\sqrt{\pi}}\int_0^x e^{-u^2}du
\eqs
\end{lem}
\begin{proof}
We consider the following double integrals:
\bq
\Xi_i=\int_{0}^{+\infty}\int_{\D}\dfrac{1}{\sqrt{2\pi\sigma_i^2}}e^{-\dfrac{(\log \Delta-\log \Delta')^2}{\sigma_i^2}}e^{-\dfrac{d_2^{2}(z,z')}{2\sigma_i^2}}\dfrac{d\Delta'}{\Delta'}\dfrac{dz_1'dz_2'}{(1-|z'|^2)^2}\quad,i=1,2,
\label{eq:xi}
\eq
so that:
\bqs
\overline{W}=\Xi_1-A\Xi_2
\eqs
Since the variables are separable, we have:
\bqs
\Xi_i=\Bigg(\int_{0}^{+\infty}\dfrac{1}{\sqrt{2\pi\sigma_i^2}}e^{-\dfrac{(\log \Delta-\log \Delta')^2}{\sigma_i^2}}\dfrac{d\Delta'}{\Delta'}\Bigg)\Bigg(\int_{\D}e^{-\dfrac{d_2^{2}(z,z')}{2\sigma_i^2}}\dfrac{dz_1'dz_2'}{(1-|z'|^2)^2}\Bigg)
\eqs
One can easily see that:
\bqs
\int_{0}^{+\infty}\dfrac{1}{\sqrt{2\pi\sigma_i^2}}e^{-\dfrac{(\log \Delta-\log \Delta')^2}{\sigma_i^2}}\dfrac{d\Delta'}{\Delta'}=\dfrac{1}{\sqrt{2}}
\eqs

We now give a simplified expression for $\Xi_i$. We set $f_i(x)=e^{-\dfrac{x^2}{2\sigma_i^2}}$ and then we have, because of lemma \ref{lem:invariance}:
\bqs
\Xi_i=\dfrac{1}{\sqrt{2}}\int_{\D}f_i(d_2(O,z'))\text{dm}(z')=\dfrac{1}{\sqrt{2}}\int_{\D}f_i(\text{arctanh}(|z'|))\text{dm}(z')
\eqs
\bqs
=\dfrac{1}{\sqrt{2}}\int_{0}^{1}\int_{0}^{2\pi}f_i(\text{arctanh}(r))\dfrac{rdrd\theta}{(1-r^2)^2}=\sqrt{2}\pi\int_{0}^{1}f_i(\text{arctanh}(r))\dfrac{rdr}{(1-r^2)^2}
\eqs
\bqs
=\sqrt{2}\pi\int_{0}^{1}e^{-\dfrac{\text{arctanh}^2(r)}{2\sigma_i^2}}\dfrac{rdr}{(1-r^2)^2}
\eqs
The change of variable $x=\text{arctanh}(r)$ implies $dx=\dfrac{dr}{1-r^2}$ and yields:
\bqs
\Xi_i=\sqrt{2}\pi\int_{0}^{+\infty}e^{-\dfrac{x^2}{2\sigma_i^2}}\dfrac{\tanh(x)}{1-\tanh^2(x)}dx=\sqrt{2}\pi\int_{0}^{+\infty}e^{-\dfrac{x^2}{2\sigma_i^2}}\sinh(x)\cosh(x)dx
\eqs
\bqs
=\dfrac{\pi}{\sqrt{2}}\int_{0}^{+\infty}e^{-\dfrac{x^2}{2\sigma_i^2}}\sinh(2x)dx=\dfrac{\pi}{2\sqrt{2}}\Bigg(\int_{0}^{+\infty}e^{-\dfrac{x^2}{2\sigma_i^2}+2x}dx-\int_{0}^{+\infty}e^{-\dfrac{x^2}{2\sigma_i^2}-2x}dx\Bigg)
\eqs
\bqs
=\dfrac{\pi}{2\sqrt{2}}e^{2\sigma_i^2}\Bigg(\int_{0}^{+\infty}e^{-\dfrac{(x-2\sigma_i^2)^2}{2\sigma_i^2}}dx-\int_{0}^{+\infty}e^{-\dfrac{(x+2\sigma_i^2)^2}{2\sigma_i^2}}dx\Bigg)
\eqs
\bqs
=\dfrac{\pi}{2}\sigma_ie^{2\sigma_i^2}\Bigg(\int_{-\sqrt{2}\sigma_i}^{+\infty}e^{-u^2}du-\int_{\sqrt{2}\sigma_i}^{+\infty}e^{-u^2}du\Bigg)=\dfrac{\pi}{2}\sigma_ie^{2\sigma_i^2}\int_{-\sqrt{2}\sigma_i}^{\sqrt{2}\sigma_i}e^{-u^2}du
\eqs
then we have a simplified expression for $\Xi_i$:
\bqs
\Xi_i=\frac{\pi^{\frac{3}{2}}}{2}\sigma_ie^{2\sigma_i^2}\textbf{erf}(\sqrt{2}\sigma_i)
\eqs

\end{proof}

\section{Isometries of $\D$}\label{section:isom}

We briefly descrbies the isometries of $\D$, i.e the transformations that preserve the distance $d_2$. We refer to the classical textbooks in hyperbolic goemetry for details, e.g, \cite{katok:92}. The direct isometries (preserving the orientation) in $\D$ are the elements of the special unitary group, noted $\text{SU}(1,1)$, of $2\times 2$ Hermitian matrices with determinant equal to $1$. Given:
\bqs
\gamma =\left(\begin{array}{ll}
 \alpha & \beta \\
 \bar\beta & \bar \alpha
\end{array}
 \right) \text{ such that } |\alpha|^2-|\beta|^2=1, 
\eqs
an element of $\text{SU}(1,1)$, the corresponding isometry $\gamma$ in $\D$ is defined by:
\bq
\gamma \cdot z=\frac{\alpha z+\beta}{\bar\beta z+\bar \alpha},\quad z\in\D
\label{eq:corresp}
\eq
Orientation reversing isometries of $\D$ are obtained by composing any transformation \eqref{eq:corresp} with the reflexion $\kappa:z\rightarrow \bar z$. The full symmetry group of the Poincar\'e disc is therefore:
\bqs
\text{U}(1,1)=\text{SU}(1,1)\cup\kappa\cdot \text{SU}(1,1)
\eqs
Let us now describe the different kinds of direct isometries acting in $\D$. We first define the following one parameter subgroups of $\text{SU}(1,1)$:
\bqs
\left\{ \begin{array}{lll}
 K\overset{def}{=}\{\text{rot}_{\phi}=\left( \begin{array}{ll}
 e^{i\frac{\phi}{2}} & 0\\
0 & e^{-i\frac{\phi}{2}}
\end{array}
\right),\phi\in \mathbb{S}^1\}\\
 A\overset{def}{=}\{a_r= \left( \begin{array}{ll}
 \cosh r & \sinh r\\
 \sinh r & \cosh r
\end{array}
\right),r\in\R\}\\
 N\overset{def}{=}\{n_s= \left( \begin{array}{ll}
 1+is & -is\\
 is & 1-is
\end{array}
\right),s\in\R\}
\end{array}
\right.
\eqs
Note that $\text{rot}_{\phi}\cdot z=e^{i\phi}z$ and also $a_r\cdot O=\tanh r$.\\
The group $K$ is the orthogonal group $\text{O}(2)$. Its orbits are concentric circles. It is possible to express each point $z\in\D$ in hyperbolic polar coordinates: $z=\text{rot}_{\phi}a_r\cdot O=\tanh r e^{i\phi}$ and $r=d_2(z,0)$.\\
 The orbits of $A$ converge to the same limit points of the unit circle $\partial\D$, $b_{\pm}=\pm$ when $r\rightarrow\pm \infty$. They are circular arcs in $\D$ going through the points $b_1$ and $b_{-1}$.\\
 The orbits of $N$ are the circles inside $\D$ and tangent to the unit circle at $b_1$. These circles are called \textit{horocycles} with base point $b_1$. $N$ is called the horocyclic group. It is also possible to express each point $z\in\D$ in horocyclic coordinates: $z=n_s a_r\cdot O$, where $n_s$ are the transformations associated with the group $N$ ($s\in\R$) and $a_r$ the transformations associated with the subroup $A$ ($r\in\R$).

\paragraph{Iwasawa decomposition}
 The following decomposition holds, see \cite{iwaniec:02}:
\bqs
\text{SU}(1,1)=KAN
\eqs
This theorem allows us to decompose any isometry of $\D$ as the product of at most thee elements in the groups, $K,A$ and $N$.\\


\begin{thebibliography}{10}

\bibitem{amari:77}
S.-I. Amari.
\newblock Dynamics of pattern formation in lateral-inhibition type neural
  fields.
\newblock {\em Biological Cybernetics}, 27(2):77--87, jun 1977.

\bibitem{bellen-zennaro:05}
A.~Bellen and M.~Zennaro.
\newblock {\em Numerical Methods for Delay Differential Equations}.
\newblock Oxford Science Publications, 2005.

\bibitem{ben-yishai-bar-or-etal:95}
R.~Ben-Yishai, RL~Bar-Or, and H.~Sompolinsky.
\newblock Theory of orientation tuning in visual cortex.
\newblock {\em Proceedings of the National Academy of Sciences},
  92(9):3844--3848, 1995.

\bibitem{bigun-granlund:87}
J.~Bigun and G.~Granlund.
\newblock Optimal orientation detection of linear symmetry.
\newblock In {\em Proc. First Int'l Conf. Comput. Vision}, pages 433--438. EEE
  Computer Society Press, 1987.

\bibitem{bressloff-cowan:03}
P.~C. Bressloff and J.~D. Cowan.
\newblock A spherical model for orientation and spatial frequency tuning in a
  cortical hypercolumn.
\newblock {\em Philosophical Transactions of the Royal Society B}, 2003.

\bibitem{bressloff-cowan:02b}
P.C. Bressloff and J.D. Cowan.
\newblock {SO}(3) symmetry breaking mechanism for orientation and spatial
  frequency tuning in the visual cortex.
\newblock {\em Phys. Rev. Lett.}, 88(7), feb 2002.

\bibitem{bressloff-cowan:02}
P.C. Bressloff and J.D. Cowan.
\newblock The visual cortex as a crystal.
\newblock {\em Physica {D}: {N}onlinear {P}henomena}, 173(3--4):226--258, dec
  2002.

\bibitem{bressloff-cowan-etal:01}
P.C. Bressloff, J.D. Cowan, M.~Golubitsky, P.J. Thomas, and M.C. Wiener.
\newblock Geometric visual hallucinations, {E}uclidean symmetry and the
  functional architecture of striate cortex.
\newblock {\em Phil. Trans. R. Soc. Lond. B}, 306(1407):299--330, mar 2001.

\bibitem{bressloff-cowan-etal:02}
P.C. Bressloff, J.D. Cowan, M.~Golubitsky, P.J. Thomas, and M.C. Wiener.
\newblock {What Geometric Visual Hallucinations Tell Us about the Visual
  Cortex}.
\newblock {\em Neural Computation}, 14(3):473--491, 2002.

\bibitem{camperi-wang:98}
M.~Camperi and X.J. Wang.
\newblock A model of visuospatial working memory in prefrontal cortex:
  Recurrent network and cellular bistability.
\newblock {\em Journal of Computational Neuroscience}, 5:383--405, 1998.

\bibitem{chossat-faugeras:09}
P.~Chossat and O.~Faugeras.
\newblock Hyperbolic planforms in relation to visual edges and textures
  perception.
\newblock {\em Plos Computational Biology}, 2009.
\newblock Accepted for publication 11/04/2009.

\bibitem{erdelyi:85}
Erdelyi.
\newblock {\em Higher Transcendental Functions}, volume~1.
\newblock Robert E. Krieger Publishing Company, 1985.

\bibitem{faugeras-veltz-etal:09}
O.~Faugeras, R.~Veltz, and F.~Grimbert.
\newblock Persistent neural states: stationary localized activity patterns in
  nonlinear continuous n-population, q-dimensional neural networks.
\newblock {\em Neural Computation}, 21(1):147--187, 2009.

\bibitem{faye-faugeras:09}
G.~Faye and O.~Faugeras.
\newblock Some theoretical and numerical results for delayed neural field
  equations.
\newblock {\em Physica D}, 2010.
\newblock Special issue on Mathematical Neuroscience.

\bibitem{folias-bressloff:04}
Stefanos~E. Folias and Paul~C. Bressloff.
\newblock Breathing pulses in an excitatory neural network.
\newblock {\em SIAM Journal on Applied Dynamical Systems}, 3(3):378--407, 2004.

\bibitem{hansel-sompolinsky:97}
D.~Hansel and H.~Sompolinsky.
\newblock Modeling feature selectivity in local cortical circuits.
\newblock {\em Methods of neuronal modeling}, pages 499--567, 1997.

\bibitem{helgason:00}
S.~Helgason.
\newblock {\em Groups and geometric analysis}, volume~83 of {\em Mathematical
  Surveys and Monographs}.
\newblock American Mathematical Society, 2000.

\bibitem{hewitt-stromberg:65}
E.~Hewitt and K.~Stromberg.
\newblock {\em Real and Abstract Analysis}, volume~25.
\newblock springer-verlag, 1965.

\bibitem{iwaniec:02}
H.~Iwaniec.
\newblock {\em Spectral methods of automorphic forms}, volume~53 of {\em {AMS}
  Graduate Series in Mathematics}.
\newblock {AMS} Bookstore, 2002.

\bibitem{katok:92}
S.~Katok.
\newblock {\em Fuchsian Groups}.
\newblock Chicago Lectures in Mathematics. The University of Chicago Press,
  1992.

\bibitem{knutsson:89}
H.~Knutsson.
\newblock Representing local structure using tensors.
\newblock In {\em Scandinavian Conference on Image Analysis}, pages 244--251,
  1989.

\bibitem{laing-troy:03}
Carlo~R. Laing and William~C. Troy.
\newblock {PDE} methods for nonlocal models.
\newblock {\em SIAM Journal on Applied Dynamical Systems}, 2(3):487--516, 2003.

\bibitem{laing-troy-etal:02}
C.L. Laing, W.C. Troy, B.~Gutkin, and G.B. Ermentrout.
\newblock Multiple bumps in a neuronal model of working memory.
\newblock {\em SIAM J. Appl. Math.}, 63(1):62--97, 2002.

\bibitem{moakher:05}
M.~Moakher.
\newblock A differential geometric approach to the geometric mean of symmetric
  positive-definite matrices.
\newblock {\em SIAM J. Matrix Anal. Appl.}, 26(3):735--747, April 2005.

\bibitem{owen-laing-etal:07}
M.R. Owen, C.R. Laing, and S.~Coombes.
\newblock Bumps and rings in a two-dimensional neural field: splitting and
  rotational instabilities.
\newblock {\em {New Journal of Physics}}, 9(10):378--401, 2007.

\bibitem{wilson-cowan:72}
H.R. Wilson and J.D. Cowan.
\newblock Excitatory and inhibitory interactions in localized populations of
  model neurons.
\newblock {\em Biophys. J.}, 12:1--24, 1972.

\end{thebibliography}
\end{document}